\documentclass[11pt]{article}
\usepackage[margin=1in]{geometry}
\usepackage[T1]{fontenc}
\usepackage[utf8]{inputenc}
\usepackage{lmodern,microtype}
\usepackage{amsmath,amssymb,amsthm,mathtools,bm,mathrsfs}
\usepackage{aliascnt}
\usepackage{booktabs,array,longtable,enumitem}
\usepackage{xcolor}
\usepackage[colorlinks=true,linkcolor=blue!42!black,citecolor=blue!42!black,urlcolor=blue!42!black]{hyperref}
\usepackage[nameinlink,capitalize,noabbrev]{cleveref}
\newtheorem{theorem}{Theorem}[section]
\newaliascnt{proposition}{theorem}
\newtheorem{proposition}[proposition]{Proposition}
\aliascntresetthe{proposition}
\newaliascnt{lemma}{theorem}
\newtheorem{lemma}[lemma]{Lemma}
\aliascntresetthe{lemma}
\newaliascnt{corollary}{theorem}
\newtheorem{corollary}[corollary]{Corollary}
\aliascntresetthe{corollary}
\theoremstyle{definition}
\newaliascnt{definition}{theorem}

\aliascntresetthe{definition}
\newaliascnt{remark}{theorem}
\newtheorem{remark}[remark]{Remark}
\aliascntresetthe{remark}

\newcommand{\R}{\mathbb R}
\newcommand{\Pp}{\mathbb P}

\newcommand{\cE}{\mathcal E}
\newcommand{\cA}{\mathcal A}
\newcommand{\Qop}{\mathcal Q}
\newcommand{\Rres}{\mathcal R}

\title{Axisymmetric Navier--Stokes with Swirl: Regularity and Blow-Up Constraints}
\author{Rishad Shahmurov\\\small Cellular Products Research and Development, Roswell, Georgia, USA}
\date{}

\begin{document}
\maketitle

\begin{abstract}
For smooth finite-energy axisymmetric Navier--Stokes solutions with swirl,
we give a continuation criterion in terms of the signed cumulative work
of the actual post-cut source-generated meridional current, with quartic
viscous dissipation counted once. In the stated strong Sobolev class the
criterion is necessary and sufficient and requires neither a global
circulation bound nor spatial moments. We also prove cumulative inward-compression criteria, collect
quantitative necessary conditions at a first singular time, and give
a complete-history mixed-residual test for noncontinuation.
Auxiliary results quantify quartic dissipation,
energy-funded source bounds, and causal source--receiver identities.
An explicit smooth whole-space datum has positive newborn source
production.
\end{abstract}



\section{Introduction}\label{sec:intro}
We consider the incompressible Navier--Stokes equations
\begin{equation}\label{eq:NS}
 \partial_tu-\nu\Delta u+(u\cdot\nabla)u+\nabla p=0,
 \qquad \nabla\cdot u=0
 \quad\text{in }\mathbb R^3\times(0,T),
\end{equation}
for smooth finite-energy axisymmetric data,
$u=u^re_r+u^\theta e_\theta+u^ze_z$.
The no-swirl case is globally regular
\cite{Ladyzhenskaya1968,UkhovskiiYudovich1968}; global regularity with
arbitrary swirl is open.

The circulation $\Gamma=ru^\theta$, the meridional vorticity variable
$G=\omega^\theta/r$, and $U=u^r/r$ obey a coupled system in which the
source $\partial_z[(u^\theta/r)^2]$ generates the meridional current.
A criterion based solely on the size of that current loses its signed
interaction with the swirl. Our first result instead uses its complete
post-cut history, including the viscous cost. The second result quantifies
cumulative inward compression; the third is a finite-prefix test for
noncontinuation.

We compare our criteria with classical circulation and critical-regularity
results \cite{ChaeLee2002,LeiZhang2011,LeiZhang2017,Wei2016,ChenFangZhang2017},
Type-I and Liouville theorems \cite{CSTY2008,KNSS2009,SereginSverak2009},
and recent radial and Euler-length conditions
\cite{Zhang2026,CIVEulerLength2026}.
The auxiliary estimates quantify radial and axial quartic dissipation,
energy-funded source norms, and exact causal source--receiver identities.
The newborn example rules out a universal negative sign for generated
source production.

\subsection*{Axisymmetric variables and equations}
We suppress the angular factor $2\pi$ and use $d\mu=r\,dr\,dz$ for the physical axisymmetric measure. The five-dimensional operator and lifted measure are
\[
 \Delta_5=\partial_r^2+\frac3r\partial_r+\partial_z^2,
 \qquad dm_5=r^3dr\,dz.
\]
The coupled vorticity estimates use $d\mu$; lifted pairings use $dm_5$ when indicated.
With $F=u^\theta/r$ and $G=\omega^\theta/r$, one has
\begin{align}
 (\partial_t+b\cdot\nabla-\nu\Delta_5)F&=-2UF,\label{eq:F}\\
 (\partial_t+b\cdot\nabla-\nu\Delta_5)G&=\partial_z(F^2),\label{eq:G}
\end{align}
where $b=u^re_r+u^ze_z$ and $U=u^r/r$.  Write the five-dimensional stream function as
\[
 -\Delta_5\psi=G,\qquad u^r=-r\psi_z,\qquad u^z=2\psi+r\psi_r,
\]
so $U=-\psi_z$.  The circulation $\Gamma=r^2F=ru^\theta$ solves
\begin{equation}\label{eq:Gamma}
 \Gamma_t+b\cdot\nabla\Gamma
 =\nu\bigl(\Gamma_{rr}-r^{-1}\Gamma_r+\Gamma_{zz}\bigr),
\end{equation}
and satisfies the circulation maximum principle whenever the initial circulation bound is finite.

\subsection*{Main results}
Throughout the main results, let $u_0\in H^m(\R^3)$, $m\ge3$, be
smooth, divergence free and axisymmetric, let $T<\infty$, and let
$u$ denote its strong solution on $[0,T)$. Fix a regular cut $a<T$
whenever required. Set
\[
 s(t)=\log\frac{T-a}{T-t},\qquad
 M(t)=\|u^\theta(t)\|_\infty,\qquad
 \Gamma_*=\|ru_0^\theta\|_\infty.
\]
For nonzero swirl define
\begin{equation}\label{eq:quartic-dissipation-def}
\begin{aligned}
 V_4(t)&=\int_{\R_+\times\R}|u^\theta(r,z,t)|^4\,d\mu,\\
 D_4(t)&=3\int |u^\theta|^2|\nabla u^\theta|^2\,d\mu
       +\int\frac{|u^\theta|^4}{r^2}\,d\mu.
\end{aligned}
\end{equation}
When $u^\theta(a)\not\equiv0$, split $G=G_{\rm old}+G_{\rm src}$
under the same actual meridional drift, with
$G_{\rm old}(a)=G(a)$ and $G_{\rm src}(a)=0$. Write
\[
 \sigma_{\rm src}(t)=-\frac1{V_4(t)}
       \int U_{\rm src}(t)|u^\theta(t)|^4\,d\mu,
 \qquad
 N_a(t)=\int_a^t\left(\sigma_{\rm src}(q)
                -\nu\frac{D_4(q)}{V_4(q)}\right)dq.
\]
For the cumulative compression statement, at almost every time choose
a derivative-realizing maximum point $x_*(t)$ of $|u^\theta|$ and put
$\sigma(t)=-U(x_*(t),t)$. The first theorem assumes no global
circulation bound or spatial moment; the second explicitly assumes
$0<\Gamma_*<\infty$.

\begin{theorem}\label{thm:signed-source-intro}
For any regular cut $a<T$ with $u^\theta(a)\not\equiv0$, smooth
continuation through $T$ is equivalent to
\begin{equation}\label{eq:signed-source-intro}
 \int_a^t e^{8N_a(q)}\,dq
 =o\!\left(\log\frac{T-a}{T-t}\right)\quad(t\uparrow T).
\end{equation}
\end{theorem}

\begin{theorem}\label{thm:cumulative-compression-intro}
Assume $0<\Gamma_*<\infty$. If
\begin{equation}\label{eq:intro-cumulative-compression}
 \limsup_{t\uparrow T}
 \frac{\displaystyle\int_a^t\sigma(q)\,dq}{s(t)}
 <\frac12+\frac{\nu^2}{6\Gamma_*^2},
\end{equation}
then the solution continues smoothly through $T$. At the endpoint
$c_*=\frac12+\frac{\nu^2}{6\Gamma_*^2}$, it is also sufficient that
\begin{equation}\label{eq:cstar-log}
 \int_a^{t(s)}\sigma(q)\,dq
 \le c_*s+d\log(1+s)+C,
 \qquad 0\le d\le\frac{1}{18(\Gamma_*/\nu)^2}.
\end{equation}
\end{theorem}

For the finite-prefix criterion, a \emph{complete approximating history}
$V_n$ on $[0,t_n]$ is smooth, globally divergence free, has
finite-energy decay, and satisfies
\[
 V_n\in C([0,t_n];L^2(\R^3))
       \cap L^2(0,t_n;\dot H^1(\R^3)),
\]
with the whole-space relative-energy calculation justified. Define
\[
 \mathcal R[V_n]=\partial_tV_n-\nu\Delta V_n
  +\Pp((V_n\cdot\nabla)V_n)=R_{a,n}+R_{b,n}
\]
and its mixed error
\[
 \varepsilon_n=\|V_n(0)-u_0\|_2
 +\|R_{a,n}\|_{L^1(0,t_n;L^2)}
 +\nu^{-1/2}\|R_{b,n}\|_{L^2(0,t_n;\dot H^{-1})}.
\]

\begin{theorem}\label{thm:mixed-prefix-intro}
Let $t_n\uparrow T$, $\ell_n\downarrow0$, and let $V_n$ be
complete approximating histories as above. Suppose
$\varepsilon_n=o(\nu\sqrt{\ell_n})$ and there are centers
$x_n\in\R^3$, one fixed $\Phi\in C_c^\infty(\R^3;\R^3)$,
and $c_0>0$ such that
\begin{equation}\label{eq:mixed-prefix-intro}
 \left|\frac{\ell_n}{\nu}\int_{\R^3}
     V_n(t_n,x_n+\ell_n y)\cdot\Phi(y)\,dy\right|\ge c_0
 \qquad\text{for all }n.
\end{equation}
Then the strong solution from $u_0$ does not continue smoothly through
$T$.
\end{theorem}

The quantitative restrictions on a first singular time are collected
in \cref{thm:necessary-blowup}. They concern compression at a swirl
maximum, angular and quartic accumulation, source inflow, and signed
quartic-receiver work, each under the corresponding continuation
criterion's hypotheses.

The proofs of the cumulative swirl and quartic-dose criteria appear in
\cref{sec:exits}. Source-inflow, source-enstrophy, and free-current
estimates follow in \cref{sec:current-quartic}; the later sections
develop source-memory bounds and the collected necessary conditions.

\section{Cumulative regularity criteria}\label{sec:exits}

\begin{theorem}\label{thm:cumulative-M-intro}
If
\begin{equation}\label{eq:intro-cumulative-M}
 \limsup_{t\uparrow T}
 \frac{\nu^{-1}\displaystyle\int_a^t M(q)^2\,dq}{s(t)}<\frac16,
\end{equation}
then the solution continues through $T$.  The endpoint estimate
\[
 \nu^{-1}\int_a^tM(q)^2\,dq
 \le \frac{s(t)}6+\frac1{18}\log(1+s(t))+C
\]
is also sufficient.
\end{theorem}

\begin{theorem}\label{thm:quartic-dose-intro}

Let $C_4<\infty$ be the explicit constant defined in \eqref{eq:C4-def}.  If
\[
 \limsup_{t\uparrow T}
 \frac{\nu^{-7}\displaystyle\int_a^tV_4(q)^2\,dq}{s(t)}
 <\frac{3}{2C_4},
\]
then the solution continues through $T$.  In particular,
\[
 \int_a^tV_4(q)^2\,dq=o(s(t))
\]
is sufficient.
\end{theorem}

Put $A(t)=\|(-U(t))_+\|_\infty$. In the coupled vorticity
energy estimates only, $J=\omega^r/r$; the centrifugal source in
\cref{sec:memory} is denoted $j=\partial_z(F^2)$.

\begin{lemma}\label{lem:JG}
For
\[
 Y(t)=\|J(t)\|_2^2+\|G(t)\|_2^2
\]
with the physical axisymmetric measure, there is a nonnegative dissipation $Z(t)$ such that
\begin{equation}\label{eq:JG}
 Y'(t)+\nu Z(t)\le \frac9\nu M(t)^2Y(t).
\end{equation}
\end{lemma}
\begin{proof}
The axisymmetric Hodge estimate gives $\|\nabla U\|_2\le \|G\|_2$.  We also use the half-weight Hardy estimate
\[
 \int \frac{f^2}{r}\,d\mu=\int f^2\,dr\,dz
 =-2\int f f_r\,d\mu
 \le2\|f\|_2\|f_r\|_2,
\]
valid for smooth decaying axisymmetric $f$; the boundary term $rf^2$ vanishes at the axis.  Cauchy--Schwarz with weight $1/r$, followed by this estimate for $J$ and $G$, gives
\[
 \int\frac{|JG|}{r}\,d\mu\le \sqrt{YZ},
\]
after the corresponding gradient terms are included in the nonnegative dissipation $Z$.

Testing the $J$ and $G$ equations by $J$ and $G$ gives
\[
 \frac12Y'+\nu Z\le I_J+I_G.
\]
Writing $v=u^\theta$, integration by parts in the $J$ interaction yields
\[
 I_J=\int v\,(J_zU_r-J_rU_z)\,d\mu,
\]
and hence
\[
 |I_J|\le M\|\nabla J\|_2\|\nabla U\|_2\le M\sqrt{YZ}.
\]
The swirl source in the $G$ equation is
\[
 I_G=-2\int FJG\,d\mu=-2\int\frac{vJG}{r}\,d\mu,
\qquad |I_G|\le2M\sqrt{YZ}.
\]
Therefore
\[
 \frac12Y'+\nu Z\le3M\sqrt{YZ}
 \le\frac\nu2Z+\frac{9}{2\nu}M^2Y,
\]
which is equivalent to \eqref{eq:JG}.
\end{proof}

The coupled $J,G$ estimate gives the cumulative criteria stated above.  We prove them first, then record stronger pointwise corollaries for comparison with the literature.

\begin{lemma}\label{lem:weighted-endpoint}
Assume
\[
 Y'(t)+\nu Z(t)\le f(t)Y(t),\qquad f\ge0,
\]
and put $H(t)=\int_a^t f(q)\,dq$ and $s=\log((T-a)/(T-t))$.  If
\begin{equation}\label{eq:weighted-endpoint-H}
 H(t)\le \frac32s+\frac12\log(1+s)+C
\end{equation}
for all sufficiently late $t$, then
\[
 \int_a^t\|U(q)\|_\infty\,dq=o(s),
\]
and the solution continues through $T$.
\end{lemma}
\begin{proof}
Multiply the differential inequality by the integrating factor $e^{-H(t)}$.  Since $H'=f$,
\[
 \frac{d}{dt}\bigl(Ye^{-H}\bigr)+\nu Ze^{-H}\le0,
\]
and integration from $a$ to $t$ gives
\begin{equation}\label{eq:weighted-if}
 Y(t)e^{-H(t)}+\nu\int_a^tZ(q)e^{-H(q)}\,dq\le Y(a).
\end{equation}
By \eqref{eq:weighted-endpoint-H},
\[
 e^{H(t)}\le C(T-t)^{-3/2}(1+s)^{1/2},
\]
so \eqref{eq:weighted-if} yields
\[
 Y(t)\le C(T-t)^{-3/2}(1+s)^{1/2}.
\]
Use $dt=(T-t)\,ds$ and define
\[
 z(s)=\nu(T-t)^{5/2}(1+s)^{-1/2}Z(t(s)).
\]
Substituting the bound for $e^H$ into the integral term of \eqref{eq:weighted-if} gives
$z\in L^1(0,\infty)$.

The axisymmetric Hodge interpolation inequality
\[
 \|U\|_\infty\le C\,Y^{1/4}Z^{1/4}
\]
therefore implies
\[
 (T-t)\|U(t(s))\|_\infty
 \le C_\nu(1+s)^{1/4}z(s)^{1/4}.
\]
For $S>L$, H\"older with exponents $4$ and $4/3$ gives
\[
 \frac1S\int_L^S(1+s)^{1/4}z(s)^{1/4}\,ds
 \le \frac{1}{S}
 \Bigl(\int_L^S z\,ds\Bigr)^{1/4}
 \Bigl(\int_L^S(1+s)^{1/3}\,ds\Bigr)^{3/4}.
\]
The last factor is $O(S)$; hence the right-hand side is
$O((\int_L^\infty z)^{1/4})$.  Let $S\to\infty$ and then $L\to\infty$ to obtain
\[
 \int_a^t\|U(q)\|_\infty\,dq=o(s(t)).
\]
Finally the maximum principle for
\[
 \partial_tu^\theta+b\cdot\nabla u^\theta-
u(\Delta-r^{-2})u^\theta=-Uu^\theta
\]
gives
$M(t)\le M(a)\exp(\int_a^t\|U\|_\infty)$, hence
$M(t)\le C_\varepsilon(T-t)^{-\varepsilon}$ for every $\varepsilon>0$.
Choose $\varepsilon<1/2$.  Then $u^\theta\in L_t^2L_x^\infty$ near $T$, and the angular Prodi--Serrin criterion of Chen--Fang--Zhang \cite{ChenFangZhang2017} gives continuation.
\end{proof}

\begin{proof}[Proof of \Cref{thm:cumulative-M-intro}]
In \eqref{eq:JG} take
\[
 f(t)=\frac9\nu M(t)^2,
 \qquad H(t)=\frac9\nu\int_a^tM(q)^2\,dq.
\]
If the strict limsup in \eqref{eq:intro-cumulative-M} is $<1/6$, choose $\varepsilon>0$ and $t_0<T$ so that, for $t\ge t_0$,
\[
 H(t)\le \Bigl(\frac32-9\varepsilon\Bigr)s(t).
\]
After enlarging the additive constant to absorb $[a,t_0]$, this implies the hypothesis \eqref{eq:weighted-endpoint-H} of \cref{lem:weighted-endpoint}.  Hence $\int_a^t\|U\|_\infty=o(s)$ and continuation follows.

For the endpoint estimate,
\[
 \nu^{-1}\int_a^tM^2\le \frac{s}{6}+\frac1{18}\log(1+s)+C
\]
gives directly
\[
 H(t)\le \frac32s+\frac12\log(1+s)+9C.
\]
Apply \cref{lem:weighted-endpoint} again.
\end{proof}

\begin{proof}[Proof of \Cref{thm:cumulative-compression-intro}]
At almost every time at which $M$ is differentiable, choose a derivative-realizing maximizer $x_*(t)=(r_*,z_*)$.  At that point $\nabla |u^\theta|=0$ and the scalar swirl equation gives
\[
 D^+M\le \sigma M-\nu\frac{M}{r_*^2}.
\]
The circulation maximum principle gives $r_*M\le\Gamma_*$, hence
\begin{equation}\label{eq:riccati-M}
 D^+M\le \sigma M-\frac{\nu}{\Gamma_*^2}M^3.
\end{equation}
Set
\[
 \tau=T-t,\qquad s=\log\frac{T-a}{\tau},\qquad
 w(s)=\frac{\tau M(t(s))^2}{\nu},\qquad
 \beta=\frac{2}{(\Gamma_*/\nu)^2},
\]
and
\[
 \mathcal C(s)=\int_a^{t(s)}\sigma(q)\,dq.
\]
Since $dt/ds=\tau$, differentiating $w$ and using \eqref{eq:riccati-M} gives the logistic/Bernoulli inequality
\[
 w_s\le\bigl(2\tau\sigma-1\bigr)w-\beta w^2.
\]
Let $B(s)=2\mathcal C(s)-s$.  The comparison equation
$\bar w_s=B'\bar w-\beta\bar w^2$ has the explicit solution
\[
 \bar w(s)=\frac{e^{B(s)}}{w(0)^{-1}+\beta\int_0^se^{B(\rho)}\,d\rho},
 \qquad w\le\bar w.
\]
Since the denominator derivative is $\beta e^{B}$,
\begin{equation}\label{eq:w-integral-bound}
 \int_0^sw(\rho)\,d\rho
 \le\frac1\beta\log\!\left(1+\beta w(0)\int_0^se^{B(\rho)}\,d\rho\right).
\end{equation}
If $\bar c=\limsup\mathcal C(s)/s$, then for every $c>\bar c$,
$B(s)\le(2c-1)s+C_c$ for large $s$.  Substitution in \eqref{eq:w-integral-bound}, followed by $c\downarrow\bar c$, gives
\[
 \limsup_{s\to\infty}\frac1s\int_0^sw\,d\rho
 \le(\Gamma_*/\nu)^2(\bar c-\tfrac12)_+.
\]
But
\[
 \int_0^sw\,d\rho=\nu^{-1}\int_a^{t(s)}M(q)^2\,dq.
\]
If $\bar c<c_*=\frac12+\frac{\nu^2}{6\Gamma_*^2}$, the last limsup is $<1/6$, so \cref{thm:cumulative-M-intro} applies.

At the endpoint, \eqref{eq:cstar-log} gives
$B(s)\le(2c_*-1)s+2d\log(1+s)+C$.  Evaluating \eqref{eq:w-integral-bound} yields
\[
 \nu^{-1}\int_a^tM^2\,dq
 \le \frac{s}{6}+(\Gamma_*/\nu)^2d\log(1+s)+C.
\]
The assumed bound on $d$ makes the logarithmic coefficient at most $1/18$; apply the endpoint part of \cref{thm:cumulative-M-intro}.
\end{proof}

\begin{proof}[Proof of \Cref{thm:quartic-dose-intro}]
Let $C_{\rm GN}$ and $C_{H,4}$ be fixed constants such that
\[
 \|f\|_4\le C_{\rm GN}\|f\|_2^{1/4}\|\nabla f\|_2^{3/4},
 \qquad
 \|\nabla U\|_4\le C_{H,4}\|G\|_4.
\]
Set
\begin{equation}\label{eq:C4-def}
 K_0=C_{\rm GN}(C_{H,4}+4),
 \qquad
 C_4=\frac{7^7}{4^8}K_0^8.
\end{equation}
The paired-vorticity interactions have the form
\[
 I_J=\int u^\theta(J_zU_r-J_rU_z)\,d\mu,
 \qquad I_G=-2\int (u^\theta/r)JG\,d\mu.
\]
We first prove the weighted radial estimate used in $I_G$.  For smooth $f$ with the regular axis behavior,
\[
 X:=\int r^{-4/3}|f|^{8/3}\,d\mu
 =-4\int r^{-1/3}|f|^{2/3}f\partial_r f\,d\mu,
\]
where the axis boundary term vanishes.  Cauchy--Schwarz and then Gagliardo--Nirenberg give
\[
 X\le4\|\nabla f\|_2X^{1/4}\|f\|_4
 \le4C_{\rm GN}\|\nabla f\|_2^{7/4}\|f\|_2^{1/4}X^{1/4}.
\]
Dividing by $X^{1/4}$ and taking the $3/8$ power yields
\begin{equation}\label{eq:weighted-radial-proof}
 \|r^{-1/2}f\|_{8/3}
 \le2C_{\rm GN}^{1/2}\|f\|_2^{1/8}\|\nabla f\|_2^{7/8}.
\end{equation}
For $I_J$, H\"older, the $L^4$ Hodge estimate, and Gagliardo--Nirenberg give
\[
 |I_J|
 \le V_4^{1/4}\|\nabla J\|_2\|\nabla U\|_4
 \le C_{H,4}C_{\rm GN}V_4^{1/4}Y^{1/8}Z^{7/8}.
\]
For $I_G$, H\"older with exponents $4,8/3,8/3$ and \eqref{eq:weighted-radial-proof} give
\[
 |I_G|
 \le2V_4^{1/4}\|r^{-1/2}J\|_{8/3}\|r^{-1/2}G\|_{8/3}
 \le4C_{\rm GN}V_4^{1/4}Y^{1/8}Z^{7/8}.
\]
Hence
\[
 |I_J|+|I_G|\le K_0V_4^{1/4}Y^{1/8}Z^{7/8}.
\]
Use Young's inequality with conjugate exponents $8/7$ and $8$:
\[
 ab\le\frac78\varepsilon^{8/7}a^{8/7}+\frac18\varepsilon^{-8}b^8,
\]
with $a=Z^{7/8}$, $b=K_0V_4^{1/4}Y^{1/8}$ and
$\varepsilon^{8/7}=4\nu/7$.  After multiplying the paired energy identity by $2$, this gives
\[
 Y'+\nu Z\le \frac{7^7}{4^8}K_0^8\nu^{-7}V_4^2Y
 =C_4\nu^{-7}V_4^2Y.
\]
Now apply \cref{lem:weighted-endpoint} with
$H=C_4\nu^{-7}\int_a^tV_4^2$.  The strict limsup condition gives \eqref{eq:weighted-endpoint-H} after a finite-time adjustment.  If instead $\int_a^tV_4^2=o(s)$, then $H=o(s)$, which is even stronger.
\end{proof}

\noindent\emph{Logarithmic example.} If $V_4(t)\le C(T-t)^{-1/2}(1+s(t))^{-\gamma}$ for some $\gamma>0$, then
\[
 \int_a^t V_4(q)^2\,dq\le C^2\int_0^{s(t)}(1+\rho)^{-2\gamma}\,d\rho=o(s(t)).
\]
Thus \cref{thm:quartic-dose-intro} applies.

The older pointwise Type-I statements below are immediate stronger sufficient conditions and are retained for comparison with the literature.

\noindent\emph{Pointwise special case.} If
\[
 \limsup_{t\uparrow T}\frac{(T-t)M(t)^2}{\nu}<\frac16,
\]
choose $m<1/6$ so that $M(t)^2\le \nu m/(T-t)$ on a terminal interval.  Integration gives
\[
 \limsup_{t\uparrow T}\frac{\nu^{-1}\int_a^tM(q)^2\,dq}{s(t)}\le m<\frac16,
\]
so \cref{thm:cumulative-M-intro} proves continuation.

\begin{corollary}\label{thm:strain}
Assume $0<\Gamma_*<\infty$.
Let
\[
 A(t)=\|(-U(t))_+\|_\infty,
 \qquad c_0=\limsup_{t\uparrow T}(T-t)A(t).
\]
If
\[
 c_0<\frac12+\frac{\nu^2}{6\Gamma_*^2},
\]
then the solution continues through $T$.  In particular $c_0\le1/2$ is regular.
\end{corollary}
\begin{proof}
At a derivative-realizing maximal-swirl point, $\sigma(t)\le A(t)$.  Hence
\[
 \limsup_{t\uparrow T}
 \frac{\int_a^t\sigma(q)\,dq}{s(t)}\le c_0.
\]
Apply \Cref{thm:cumulative-compression-intro}.
\end{proof}

\section{Source-generated current and quartic response}\label{sec:current-quartic}

Fix a regular cut $a<T$ and write
\[
 S_h=e^{\nu h\Delta_5},\qquad h\ge0.
\]
Let $H$ denote the homogeneous current transported by the \emph{actual} meridional drift,
\[
 (\partial_t+b\cdot\nabla-\nu\Delta_5)H=0,\qquad H(a)=G(a),
\]
and define the zero-born recent current
\[
 N=G-H.
\]
Then
\begin{equation}\label{eq:N-equation}
 (\partial_t+b\cdot\nabla-\nu\Delta_5)N=\partial_z(F^2),
 \qquad N(a)=0.
\end{equation}
The point of the decomposition below is not to replace the physical equation by a heat equation.  It isolates the part of the recent current generated by the free heat evolution of the swirl at the cut and leaves every nonlinear correction in one exact remainder.

Set
\[
 F_1(t)=S_{t-a}F(a),\qquad f(t)=F(t)-F_1(t),
\]
and define the freely generated first current
\begin{equation}\label{eq:G2def}
 G_2(t)=\int_a^tS_{t-s}\,\partial_z(F_1(s)^2)\,ds.
\end{equation}
Its meridional Hodge field is
\[
 U_2=-\partial_zP_5G_2,\qquad B_2^r=rU_2,
 \qquad P_5=(-\Delta_5)^{-1}.
\]

\begin{theorem}\label{thm:first-current-intro}
Assume $0<\Gamma_*<\infty$. After every regular cut $a<T$, the free first current $B_2$, defined in the proof below, satisfies
\begin{equation}\label{eq:first-current-bound}
 \sqrt{\frac{t-a}{\nu}}\,\|B_2^r(t)\|_\infty
 \le C_{\rm heat}\left(\frac{\Gamma_*}{\nu}\right)^2,
 \qquad a<t<T.
\end{equation}
The exact difference between the actual zero-born current and this free
response is the nonlinear Duhamel remainder \eqref{eq:recent-remainder}.
\end{theorem}

Thus a radial current exceeding the free elapsed Type-I scale cannot come
from this free response alone: it requires the nonlinear remainder,
the inherited current, or both.

\begin{proposition}\label{prop:recent-remainder}
For every $a<t<T$,
\begin{equation}\label{eq:recent-remainder}
\boxed{
 N(t)-G_2(t)
 =
 \int_a^t S_{t-s}
 \Big[
   \partial_z\!\big(2F_1f+f^2\big)-b\cdot\nabla N
 \Big](s)\,ds .}
\end{equation}
Thus the source modification $2F_1f+f^2$ and the transport of the zero-born current by the actual drift are both retained.
\end{proposition}

\begin{proof}
Equation \eqref{eq:N-equation} is
\[
 \partial_tN-\nu\Delta_5N
 =\partial_z(F^2)-b\cdot\nabla N,\qquad N(a)=0.
\]
Duhamel's formula gives
\[
 N(t)=\int_a^tS_{t-s}
 \big[\partial_z(F^2)-b\cdot\nabla N\big](s)\,ds.
\]
Since $F=F_1+f$,
\[
 F^2=F_1^2+2F_1f+f^2.
\]
Subtracting \eqref{eq:G2def} proves \eqref{eq:recent-remainder}.  The transported quantity is $N$, not the total current $G$, because the inherited homogeneous field has already been separated into $H$.
\end{proof}

At the original cut, the source tag vanishes.  Thus its physical quartic
work $\sigma_{\rm src}(a)=0$ has the exact newborn derivative
\begin{equation}\label{eq:newborn-source-slope}
 \sigma_{\rm src}'(a)
 =-\frac1{V_4(a)}\int r^2F(a)^4
 \big[-\partial_{zz}(-\Delta_5)^{-1}\big](F(a)^2)
 \,dm_5.
\end{equation}
This derivative does not depend on the inherited current at the cut;
source-age transport and the evolving receiver enter at higher order.
No sign of \eqref{eq:newborn-source-slope} is imposed.

\begin{proposition}\label{prop:newborn-curvature-r96}
Fix a regular cut $a$ with $V_4(a)>0$ and suppose $F(a)$ is
axis-compatible and sufficiently smooth and decaying for the integrations
below (in particular, this holds for smooth rapidly decaying profiles).
Set
\[
 \Gamma=r^2F(a),\qquad
 \chi=P_5\partial_z\bigl(F(a)^2\bigr),\qquad
 \Delta_{-1}=\partial_r^2-r^{-1}\partial_r+\partial_z^2.
\]
Then the exact newborn slope satisfies
\begin{equation}\label{eq:newborn-curvature-r96}
 \sigma_{\rm src}'(a)
 =-\frac{2}{V_4(a)}\int\Gamma^2|\nabla\chi|^2\,d\mu
  +\frac{1}{V_4(a)}\int\Delta_{-1}(\Gamma^2)\chi^2\,d\mu.
\end{equation}
In particular, $\sigma_{\rm src}'(a)>0$ requires
\begin{equation}\label{eq:newborn-curvature-necessary-r96}
 \int\Delta_{-1}(\Gamma^2)\chi^2\,d\mu
 >2\int\Gamma^2|\nabla\chi|^2\,d\mu.
\end{equation}
\end{proposition}
\begin{proof}
At the cut, $G_{\rm src}(a)=0$ and
$\partial_tG_{\rm src}(a)=\partial_z(F(a)^2)$; the transport term vanishes
for this tagged current. Thus
$\partial_tU_{\rm src}(a)=-\partial_z\chi$, giving
\[
 V_4(a)\sigma_{\rm src}'(a)
  =\int r^2F(a)^4\partial_z\chi\,dm_5,
 \qquad dm_5=r^3\,dr\,dz.
\]
Writing $q=r^2F(a)^2=\Gamma^2/r^2$, using
$-\Delta_5\chi=\partial_z(F(a)^2)$, and integrating first in $z$ and
then against the five-dimensional radial measure yields
\[
 V_4(a)\sigma_{\rm src}'(a)
 =-2\int q\chi\partial_z(F(a)^2)\,dm_5
 =2\int q\chi\Delta_5\chi\,dm_5
 =-2\int q|\nabla\chi|^2\,dm_5
  +\int(\Delta_5q)\chi^2\,dm_5.
\]
The elementary intertwining identity
\[
 r^3\Delta_5(\Phi/r^2)=r\Delta_{-1}\Phi
\]
and $q\,dm_5=\Gamma^2\,d\mu$ give \eqref{eq:newborn-curvature-r96}.
\end{proof}

The first term of \eqref{eq:newborn-curvature-r96} is nonpositive,
but the second need not be. The strictly positive example in
\cref{app:positive-newborn-r96} excludes universal newborn braking.

\begin{proposition}\label{prop:bounded-fixed-cut-tags}
Fix a positive regular strong cut $a<T$ such that
$G(a),F_z(a)\in L^2(\R^3)\cap L^\infty(\R^3)$ and
$F(a)\in L^\infty(\R^3)$, with axis-compatible smoothness.
Write $U_{\rm old}=-\partial_zP_5H$ and
$\mathcal U=-\partial_zP_5(N-G_2)$.  Then
\begin{equation}\label{eq:bounded-fixed-linear-tags}
 \sup_{a\le t<T}\|U_{\rm old}(t)\|_\infty\le C_a,
 \qquad \|U_2(t)\|_\infty\le C_a(t-a).
\end{equation}
Moreover the original signed source action can be written
\begin{equation}\label{eq:fixed-remainder-action}
 N_a(t)=\int_a^t\left[
 -\langle\mathcal U(q)\rangle_{\pi_4(q)}
 -\nu\frac{D_4(q)}{V_4(q)}\right]dq+B_a(t),
 \qquad |B_a(t)|\le\frac12 C_a(T-a)^2.
\end{equation}
Consequently the signed continuation criterion in
\cref{thm:signed-source-intro} is unchanged if $N_a$ is replaced by
the integral in brackets.  This is a fixed-cut reduction, not a new
continuation estimate.
\end{proposition}
\begin{proof}
For a smooth scalar $q$ solving
$(\partial_t+b\cdot\nabla-\nu\Delta_5)q=0$,
physical incompressibility and radial integration by parts give, for
$2\le p<\infty$,
\[
 \frac1p\frac d{dt}\|q\|_p^p
 +\frac{4\nu(p-1)}{p^2}\|\nabla|q|^{p/2}\|_2^2
 +\frac{4\pi\nu}{p}\int_\R|q(0,z)|^p\,dz=0.
\]
In particular the physical $L^2$ norm contracts; the smooth lifted
maximum principle controls the $L^\infty$ norm.  These statements apply
to $H$ under the \emph{actual} drift and to the free semigroup $S_h$.
For the whole-space axisymmetric radial Hodge operator $K=-\partial_zP_5$,
the classical physical Biot--Savart bound, split into near and far parts,
gives
\[
 \|Kq\|_\infty
 \le C\|q\|_2^{2/3}\|q\|_\infty^{1/3}.
\]
Applying it to the homogeneous tag yields the first bound in
\eqref{eq:bounded-fixed-linear-tags}.  Since
$F_1(t)=S_{t-a}F(a)$ and $\partial_z$ commutes with the free
semigroup, physical $L^p$ contraction gives, for $p=2,\infty$,
\[
 \|G_2(t)\|_p
 \le 2(t-a)\|F(a)\|_\infty\|F_z(a)\|_p.
\]
The same Hodge estimate gives the second bound.  Finally,
$U_{\rm src}=U_2+\mathcal U$ in the \emph{same original source cut}.
Substitution into the definition of $N_a$ gives
$B_a(t)=-\int_a^t\langle U_2(q)\rangle_{\pi_4(q)}dq$,
which satisfies the stated bound.  Thus the exponential of the two
actions differs by a factor bounded above and below by positive
constants depending on the fixed cut.  The remaining field
$\mathcal U$ still depends on the full actual drift and inherited-current
shear through \eqref{eq:recent-remainder}; no source is reset and no
velocity derivative is controlled by this proposition.
\end{proof}

Put
\begin{equation}\label{eq:Cheat}
 C_{\rm heat}
 =
 \frac{\sqrt{2/(\pi e)}}{8\sqrt e}\,
 \mathcal J,\qquad
 \mathcal J=
 \int_0^1\int_{1-\sigma}^{\infty}
 \frac{\log(1+2\alpha/\sigma)}{\alpha^{5/2}}
 \,d\alpha\,d\sigma.
\end{equation}
The integral is finite.  Numerically $C_{\rm heat}\approx0.19608063$, but no sharpness is claimed.

No single-core, finite-scale-depth, axial-compactness, or sign-alignment hypothesis is used in \Cref{thm:first-current-intro}.

\begin{proof}[Proof of \Cref{thm:first-current-intro}]
Translate time so that $a=0$ and put $h=t-a$.  Positivity of the heat semigroup and the four-dimensional radial heat identity give
\[
 |F_1(s,r,z)|
 \le \Gamma_*\,g_{\nu s}(r),
 \qquad
 g_q(r)=e^{q\Delta_4}(r^{-2})
       =\frac{1-e^{-r^2/(4q)}}{r^2}.
\]
For $q>0$,
\[
 g_q(r)=\int_0^{1/(4q)}e^{-\eta r^2}\,d\eta.
\]
Let
\[
 E(A,q,r)=e^{A\Delta_4}(g_q^2)(r).
\]
Squaring the Gaussian representation, applying the four-dimensional Gaussian heat formula, and writing $v=\eta+\xi$ gives
\[
 E(A,q,r)=
 \int_0^{1/(2q)}
 \frac{m_q(v)}{(1+4Av)^2}
 \exp\!\left(-\frac{vr^2}{1+4Av}\right)\,dv,
\]
where $0\le m_q(v)\le v$.  From $xe^{-x}\le e^{-1}$,
\[
 r^2E(A,q,r)\le
 \frac{\log(1+2A/q)}{4eA}.
\]
Discarding the exponential and integrating $v(1+4Av)^{-2}$ gives
\[
 E(A,q,r)\le
 \frac{\log(1+2A/q)}{16A^2}.
\]
Taking the geometric mean,
\begin{equation}\label{eq:gaussian-mix}
 rE(A,q,r)\le
 \frac{\log(1+2A/q)}{8\sqrt e\,A^{3/2}}.
\end{equation}

Since $P_5=(-\Delta_5)^{-1}$,
\[
 P_5e^{\nu(h-s)\Delta_5}
 =\int_{\nu(h-s)}^\infty e^{A\Delta_5}\,dA.
\]
Commuting $\partial_z$ with the heat semigroup gives
\[
 U_2(h)
 =-\int_0^h\int_{\nu(h-s)}^\infty
 \partial_{zz}e^{A\Delta_5}(F_1(s)^2)\,dA\,ds.
\]
The one-dimensional axial heat kernel satisfies
\[
 \|\partial_{zz}H_A^{(1)}\|_{L^1(\R)}
 =\frac{\sqrt{2/(\pi e)}}{A}.
\]
Using \eqref{eq:gaussian-mix} and $|F_1|^2\le\Gamma_*^2g_{\nu s}^2$,
\[
 r|U_2(h,r,z)|
 \le
 \frac{\sqrt{2/(\pi e)}}{8\sqrt e}\,
 \Gamma_*^2
 \int_0^h\int_{\nu(h-s)}^\infty
 \frac{\log(1+2A/(\nu s))}{A^{5/2}}
 \,dA\,ds.
\]
Set $s=h\sigma$ and $A=\nu h\alpha$.  The last expression is
\[
 C_{\rm heat}\,
 \frac{\Gamma_*^2}{\nu^{3/2}h^{1/2}}.
\]
Since $B_2^r=rU_2$, this is \eqref{eq:first-current-bound}.
\end{proof}

\begin{remark}\label{rem:two-clocks}
The clock in \eqref{eq:first-current-bound} is $t-a$, not $T-t$.  For a fixed old cut $a<T$ the free first current is harmless on the terminal scale.  For a rolling decomposition with $t-a$ comparable to $T-t$, the estimate has the correct terminal Type-I order.  No pointwise estimate for the inherited homogeneous field is inferred.
\end{remark}

The combination of \cref{prop:recent-remainder,thm:first-current-intro} identifies the first unpriced recent-current mechanism.  Any radial inflow worse than the free elapsed Type-I scale must be supplied by the nonlinear remainder \eqref{eq:recent-remainder}, by the inherited homogeneous field, or through their signed interaction in the receiver below.

\subsection*{Localized source-generated inward displacement}

Fix a regular cut $a<T$ and split $G=G_{\rm old}+G_{\rm src}$,
where $G_{\rm old}$ evolves homogeneously from $G(a)$ and $G_{\rm src}$
is generated after $a$ by the actual source $\partial_z(F^2)$ under the
same meridional drift. Let $u_{\rm old}^r,u_{\rm src}^r$ be their Hodge
radial velocities. For $R>0$, set
\[
 g_R^{\rm src}(t)=\sup_{0<r<R,\ z}(-u_{\rm src}^r(r,z,t))_+.
\]

\begin{theorem}\label{thm:source-inflow-intro}
Assume $0<\Gamma_*<\infty$. If for every interval $I\subset[a,T)$,
\begin{equation}\label{eq:intro-source-inflow}
 \int_I g_R^{\rm src}(t)\,dt\le A\sqrt{\nu|I|}
\end{equation}
with a fixed $A<\infty$, then the solution continues through $T$.
\end{theorem}

The inherited radial current is estimated separately and becomes
negligible in the normalized short-window displacement.

\begin{proof}[Proof of \Cref{thm:source-inflow-intro}]
We first state the comparison for the total radial velocity.  Put
\[
 g_R(t)=\sup_{0<r<R,z}(-u^r(r,z,t))_+.
\]
After rescaling time to $\nu=1$, let $q$ solve
\[
 q_t=q_{rr}+g_R(t)q_r,\qquad q(0,t)=0,\quad q(R,t)=1,
\]
with initial data $q(r,a)=\min(1,kr)$, where $k$ dominates the initial circulation profile.  Since $q_r\ge0$, the circulation operator applied to $\Gamma_*q$ has the correct supersolution sign.  The circulation maximum principle therefore gives
\[
 |\Gamma(r,z,t)|\le \Gamma_*q(r,t),\qquad 0<r<R.
\]

Let
\[
 X(t)=\int_a^t g_R(q)\,dq,
 \qquad \varphi(x,t)=q(x-X(t),t).
\]
Then $\varphi_t=\varphi_{xx}$ on the moving interval
$X(t)<x<X(t)+R$, with $\varphi(X(t),t)=0$.  The interval hypothesis implies the $1/2$-H\"older bound
\[
 |X(t)-X(s)|\le A\sqrt{|t-s|}.
\]
Fix a dyadic radius $\rho_j=2^{-j}R$.  During a time interval of length $c\rho_j^2$, the left boundary moves by at most $A\sqrt c\,\rho_j$.  Choosing $c=c(A)>0$ small, every rescaled space-time box contains a fixed interior parabolic box, independent of $j$.  The parabolic maximum principle together with the interior Harnack/oscillation lemma for the heat equation therefore gives a scale-independent number $0<\theta_A<1$ such that
\[
 m_{j+1}\le\theta_A m_j,
 \qquad
 m_j:=\sup_{0<x-X(t)\le\rho_j}\varphi(x,t),
\]
for all sufficiently late scales.  Iteration yields
\[
 m_j\le C\theta_A^j=C\rho_j^\alpha,
 \qquad \alpha=-\log_2\theta_A>0.
\]
Returning to $\Gamma$ gives
\[
 |\Gamma(r,z,t)|\le Cr^\alpha
\]
uniformly near the axis for late times.  Since every positive power dominates the logarithmic modulus $|\log r|^{-3/2}$ as $r\downarrow0$, Wei's criterion \cite{Wei2016} gives continuation.

For the source-only hypothesis, write
\[
 g_R(t)\le g_R^{\rm src}(t)+R\|U_{\rm old}(t)\|_\infty.
\]
For the homogeneous current, the $L^2$ energy identity gives
\[
 \sup_{t\ge a}\|G_{\rm old}(t)\|_2^2
 +2\nu\int_a^T\|\nabla G_{\rm old}(t)\|_2^2\,dt
 \le C(a).
\]
The Hodge interpolation estimate
\[
 \|U_{\rm old}\|_\infty
 \le C\|G_{\rm old}\|_2^{1/2}\|\nabla G_{\rm old}\|_2^{1/2}
\]
and H\"older in time therefore give, for every interval $I$,
\[
 \int_I R\|U_{\rm old}(t)\|_\infty\,dt
 \le C(a,R,\nu)|I|^{3/4}
 =o(\sqrt{\nu|I|})
 \quad(|I|\downarrow0).
\]
Hence the inherited current changes only the constant in the moving-boundary estimate and is negligible on the normalized short windows.  The preceding comparison applies to the total inflow, so continuation follows.
\end{proof}

\begin{remark}\label{rem:source-inflow-spikes}
The theorem does not require $(-u_{\rm src}^r)_+\lesssim (T-t)^{-1/2}$ pointwise.  Arbitrarily high spikes are allowed if their time integral on every window remains at most a fixed multiple of $\sqrt{\nu|I|}$.  Conversely, any singular history must violate this diffusive displacement bound on arbitrarily short windows after every fixed regular cut.
\end{remark}

\begin{proof}[Proof of \cref{thm:signed-source-intro}]
The old current solves the homogeneous transport--diffusion equation with
the \emph{actual} drift.  The physical $L^2$ energy and Hodge--Agmon
estimates, as used above, give
\[
 \int_a^T\|U_{\rm old}(t)\|_\infty\,dt<\infty.
\]
They depend on the strong Sobolev datum at the fixed cut, but not on
$\|\Gamma(a)\|_\infty$.  The exact quartic balance and the decomposition
$U=U_{\rm old}+U_{\rm src}$ yield
\[
 \frac14\log\frac{V_4(t)}{V_4(a)}
  =N_a(t)-\int_a^t\langle U_{\rm old}(q)\rangle_{\pi_4(q)}\,dq.
\]
The absolute value of the last integral is uniformly bounded by
$C_a=\int_a^T\|U_{\rm old}\|_\infty$.  Therefore
\[
 V_4(t)^2\le V_4(a)^2e^{8C_a}e^{8N_a(t)}.
\]
The hypothesis of the theorem gives
$\int_a^tV_4(q)^2dq=o(s(t))$.  The proof of
\cref{thm:quartic-dose-intro} uses only the coupled $J,G$ energy,
physical Hodge/Sobolev estimates, and the angular regularity criterion;
it requires no bound on global circulation.  Applying it proves
continuation.  Conversely, if the solution extends smoothly through $T$,
then $V_4$ is bounded on the compact terminal interval.  The identity
above and the old-current bound show that $e^{8N_a}$ is bounded there;
its time integral is finite, as asserted.
\end{proof}

\begin{corollary}\label{cor:sourced-enstrophy-intro}
Under the hypotheses of \cref{thm:signed-source-intro}, put
\[
 Z_{\rm src}(t)=\int |G_{\rm src}(t,r,z)|^2r^3\,dr\,dz.
\]
There is an absolute constant $C_{\rm s}<\infty$ such that
\begin{equation}\label{eq:sourced-enstrophy-intro}
 N_a(B)\le C_{\rm s}\nu^{-3}\int_a^B Z_{\rm src}(q)^2\,dq.
\end{equation}
Consequently $\int_a^T Z_{\rm src}(t)^2dt<\infty$ implies smooth
continuation.  One may take $C_{\rm s}=4/(81\pi^2)$ with the reduced
measure above.  For actual strong solutions on a fixed finite horizon, smooth
continuation also implies this integrability; thus it is an equivalent
continuation condition in the stated class.  Its unsigned majorant is
nevertheless stronger as an \emph{a priori estimate} than the signed
criterion and discards potentially important cancellation.
\end{corollary}

\begin{proof}[Proof of \cref{cor:sourced-enstrophy-intro}]
Let $w=(u^\theta)^2$.  The full-space Sobolev inequality
$\|f\|_6\le S_3^{1/2}\|\nabla f\|_2$, where
$S_3=4^{2/3}/(3\pi^{4/3})$, and H\"older interpolation give
\[
 |\sigma_{\rm src}|
 \le\frac{\|b_{\rm src}\|_6\|w/r\|_2\|w\|_3}
               {\|w\|_2^2}
 \le S_3^{3/4}\|\nabla b_{\rm src}\|_2
    \left(\frac{\|w/r\|_2^2}{\|w\|_2^2}\right)^{1/2}
    \left(\frac{\|\nabla w\|_2^2}{\|w\|_2^2}\right)^{1/4}.
\]
The norms in this display are full three-dimensional norms.  Put
\[
 X=\frac{\|w/r\|_2^2}{\|w\|_2^2},\qquad
 Y=\frac34\frac{\|\nabla w\|_2^2}{\|w\|_2^2}.
\]
Then $a_4=D_4/V_4=X+Y$, and optimization at $X=2Y$ yields
$X^{1/2}(4Y/3)^{1/4}\le(2/3)a_4^{3/4}$.  The whole-space
curl--gradient identity for the divergence-free meridional field gives
$\|\nabla b_{\rm src}\|_{2,\mathrm{full}}^2=2\pi Z_{\rm src}$.
Hence
\[
 \sigma_{\rm src}\le
 \frac{4\sqrt2}{3^{7/4}\sqrt\pi}
       \sqrt{Z_{\rm src}}\,a_4^{3/4}.
\]
For $q\ge0$ and $x\ge0$, Young's inequality in optimized form is
\[
 qx^{3/4}-\nu x\le\frac{27q^4}{256\nu^3}.
\]
Substitute $q=4\sqrt2\sqrt{Z_{\rm src}}/(3^{7/4}\sqrt\pi)$ and
$x=a_4$.  This gives
\[
 N_a'(t)\le\frac4{81\pi^2\nu^3}Z_{\rm src}(t)^2.
\]
Integrating and applying \cref{thm:signed-source-intro} proves the
corollary.  This direct Sobolev constant is sufficient; no claim of
optimality is made.  The estimated source current is the genuine
post-cut Hodge field and the quartic dissipation is charged once.

For the converse, suppose the actual strong solution extends smoothly through
$T$.  Put $\eta_{\rm old}=rG_{\rm old}$.  Its homogeneous equation gives
\[
 (\partial_t+b\cdot\nabla)\eta_{\rm old}
 =\nu(\Delta_3-r^{-2})\eta_{\rm old}+U\eta_{\rm old}.
\]
The physical energy identity and boundedness of $U$ on the compact smooth
horizon give
\[
 \sup_{a\le t\le T}\|\eta_{\rm old}(t)\|_2<\infty.
\]
The physical vorticity $\omega^\theta=rG$ is also uniformly bounded
in $L^2$.  Hence $rG_{\rm src}=\omega^\theta-\eta_{\rm old}$
is uniformly bounded in $L^2$, and
$Z_{\rm src}(t)=(2\pi)^{-1}\|rG_{\rm src}(t)\|_2^2$
is bounded.  Since $T-a<\infty$, its squared time integral is finite.
This converse uses smooth continuation only in the necessity direction,
not in the estimate leading to continuation.
\end{proof}

\subsection*{The critical $L^4$-weighted compression receiver}
Let $v=u^\theta$ and
\[
 \Delta_3=\partial_r^2+\frac1r\partial_r+\partial_z^2,
 \qquad P_5=(-\Delta_5)^{-1},
 \qquad \mathcal A=-\partial_zP_5.
\]
Then $U=\mathcal A G$.  Define
\[
 \mathcal M h:=\mathcal A\partial_zh=-\partial_zP_5\partial_zh,
 \qquad
 \mathcal C:=[b\cdot\nabla,\mathcal A]G.
\]
Since $\mathcal A$ commutes with $\Delta_5$, applying $\mathcal A$ to \eqref{eq:G} gives
\begin{equation}\label{eq:Ut}
 (\partial_t+b\cdot\nabla)U
 =\nu\Delta_5U+\mathcal M(F^2)+\mathcal C.
\end{equation}

Assume $v$ is not identically zero.  Put
\[
 V_4(t)=\int |v|^4\,d\mu,
 \qquad
 d\pi_4=\frac{|v|^4}{V_4}\,d\mu,
 \qquad
 \sigma_4(t)=-\int U\,d\pi_4.
\]
Thus $\pi_4$ is a probability measure.  Define $d_4$ $\pi_4$-almost everywhere by
\[
 d_4=3\frac{|\nabla v|^2}{|v|^2}+\frac1{r^2},
\]
with all weighted products interpreted through
\[
 |v|^4d_4=3|v|^2|\nabla v|^2+\frac{|v|^4}{r^2}.
\]

\begin{proposition}\label{prop:sigma4}
At every smooth preterminal time with $V_4>0$,
\begin{equation}\label{eq:sigma4prime}
\boxed{
\begin{aligned}
 \sigma_4'={}&-\mathbb E_{\pi_4}[\mathcal M(F^2)+\mathcal C]
 +4\operatorname{Var}_{\pi_4}(U)\\
 &-\nu\mathbb E_{\pi_4}[(\Delta_5+\Delta_3)U]
 +4\nu\operatorname{Cov}_{\pi_4}(U,d_4).
\end{aligned}}
\end{equation}
\end{proposition}

\begin{proof}
Write $D_t=\partial_t+b\cdot\nabla$ and $w=|v|^4$.  The swirl equation is
\[
 D_tv=-Uv+\nu(\Delta_3-r^{-2})v.
\]
Using $\Delta_3(v^4)=4v^3\Delta_3v+12v^2|\nabla v|^2$ on the sign components of $v$ and then approximating across its zero set gives
\begin{equation}\label{eq:w-evol}
 D_tw=-4Uw+\nu\Delta_3w-4\nu d_4w.
\end{equation}
The meridional field is divergence free with respect to $d\mu=r\,dr\,dz$, so material transport contributes no boundary term.  Hence
\begin{equation}\label{eq:V4rate}
 \frac{V_4'}{V_4}
 =-4\mathbb E_{\pi_4}[U]-4\nu\mathbb E_{\pi_4}[d_4]
 =4\sigma_4-4\nu\mathbb E_{\pi_4}[d_4].
\end{equation}
Let $I_4=\int Uw\,d\mu$, so $\sigma_4=-I_4/V_4$.  Combining \eqref{eq:Ut} and \eqref{eq:w-evol} yields
\[
 \begin{aligned}
 I_4'={}&\int[\mathcal M(F^2)+\mathcal C+\nu\Delta_5U]w\,d\mu
 -4\int U^2w\,d\mu\\
 &+\nu\int U\Delta_3w\,d\mu-4\nu\int Ud_4w\,d\mu.
 \end{aligned}
\]
By self-adjointness of $\Delta_3$ in the physical axisymmetric measure,
\[
 \int U\Delta_3w\,d\mu=\int w\Delta_3U\,d\mu.
\]
Now differentiate $\sigma_4=-I_4/V_4$ and use \eqref{eq:V4rate}.  The $U^2$ terms combine into $4\operatorname{Var}_{\pi_4}(U)$ and the $d_4$ terms into $4\nu\operatorname{Cov}_{\pi_4}(U,d_4)$.  The remaining terms are exactly those displayed in \eqref{eq:sigma4prime}.
\end{proof}

Fix a regular cut $a<T$ and set
\[
 \tau=T-t,
 \qquad
 s=\log\frac{T-a}{\tau},
 \qquad
 \widehat\sigma_4(s)=\tau\sigma_4(t(s)).
\]
Define
\begin{equation}\label{eq:F4}
\boxed{
\begin{aligned}
 \mathscr F_4(s)=\tau^2\Big\{&-\mathbb E_{\pi_4}[\mathcal M(F^2)+\mathcal C]
 +4\operatorname{Var}_{\pi_4}(U)\\
 &-\nu\mathbb E_{\pi_4}[(\Delta_5+\Delta_3)U]
 +4\nu\operatorname{Cov}_{\pi_4}(U,d_4)\Big\}.
\end{aligned}}
\end{equation}
Since $ds/dt=1/\tau$, \eqref{eq:sigma4prime} is equivalent to
\begin{equation}\label{eq:sigmahat}
 \boxed{\partial_s\widehat\sigma_4+\widehat\sigma_4=\mathscr F_4.}
\end{equation}

\begin{theorem}\label{thm:q4receiver}
If
\begin{equation}\label{eq:F4criterion}
 \limsup_{S\to\infty}\frac1S\int_0^S[\mathscr F_4(s)]_+\,ds<\frac18,
\end{equation}
then the solution continues through $T$.
\end{theorem}

\begin{proof}
Variation of constants in \eqref{eq:sigmahat} gives
\[
 \widehat\sigma_4(s)
 \le e^{-s}[\widehat\sigma_4(0)]_+
 +\int_0^s e^{-(s-\xi)}[\mathscr F_4(\xi)]_+\,d\xi.
\]
Integrating in $s$ and changing the order of integration yields
\begin{equation}\label{eq:sigma-int-bound}
 \int_0^S\widehat\sigma_4(s)\,ds
 \le [\widehat\sigma_4(0)]_++\int_0^S[\mathscr F_4(s)]_+\,ds.
\end{equation}
The exact $L^4$ swirl identity also gives
\[
 \frac d{dt}\log\|v(t)\|_4\le\sigma_4(t),
 \qquad
 \frac d{ds}\log\|v(t(s))\|_4\le\widehat\sigma_4(s).
\]
Choose $c$ strictly between the limsup in \eqref{eq:F4criterion} and $1/8$.  Equation \eqref{eq:sigma-int-bound} gives, for all sufficiently large $S$,
\[
 \|v(t(S))\|_4\le Ce^{cS}=C(T-t(S))^{-c}.
\]
Because $dt=(T-a)e^{-s}\,ds$,
\[
 \int_a^T\|u^\theta(t)\|_4^8\,dt
 \le C\int_0^\infty e^{(8c-1)s}\,ds<\infty.
\]
Thus $u^\theta\in L_t^8L_x^4$ on the terminal interval, and the angular Prodi--Serrin criterion of Chen--Fang--Zhang \cite{ChenFangZhang2017} gives continuation.
\end{proof}

The necessary inequality corresponding to this theorem is included in
\eqref{eq:nec-F4}.

\begin{remark}
The value $1/8$ is the critical time-growth exponent associated with the angular endpoint $L_t^8L_x^4$.  The criterion is not asserted to be a sharp regularity threshold.  Its importance is structural: the direct source, transport commutator, variance and viscous terms enter through one exact signed forcing $\mathscr F_4$.  A termwise absolute estimate is sufficient but may be much stronger than the receiver actually consumes.
\end{remark}

For later source-history estimates it is useful to split the source without imposing reflection symmetry. Fix $a<t<T$, put
\[
 h=F^2,\qquad s_+=(\partial_zh)_+,\qquad s_-=(-\partial_zh)_+,
\]
and solve, with the actual drift $b$,
\[
 (\partial_t+b\cdot\nabla-\nu\Delta_5)q_\pm=s_\pm,
 \qquad q_\pm(a)=0.
\]
Then $q_\pm\ge0$ by the parabolic maximum principle and
\[
 G_{\rm src}=q_+-q_-.
\]
Let $-\Delta_5\Psi_\pm=q_\pm$, $\Psi_\pm\ge0$, and $U_{\rm src}=-\partial_z(\Psi_+-\Psi_-)$. Then
\begin{equation}\label{eq:source-sign-split}
\boxed{
 -\int U_{\rm src}v^4\,d\mu
 =2\int r^2h\Bigl[\Psi_-s_+ +\Psi_+s_--\Psi_+s_+-\Psi_-s_-\Bigr]\,dm_5 .}
\end{equation}
Indeed,
\[
 \int \partial_z(\Psi_+-\Psi_-)r^2h^2\,dm_5
 =-2\int(\Psi_+-\Psi_-)r^2h\,\partial_zh\,dm_5,
\]
and $\partial_zh=s_+-s_-$. Thus opposite-sign and same-sign source interactions enter with opposite signs; no sign of their difference is imposed. With a smooth spatial cutoff $\chi$, integration by parts adds
\[
 -\int(\Psi_+-\Psi_-)r^2h^2\partial_z\chi\,dm_5.
\]
A $\chi_t$ term appears only when differentiating a moving local stock.

\medskip\noindent\emph{Inherited-current covariance.}
For the fixed cut $a$, let $U=U_{\rm old}+U_{\rm src}$ and
$\sigma_4^{\rm src}=-\mathbb E_{\pi_4}[U_{\rm src}]$.
Differentiating this source-weighted receiver involves the exact algebraic split
\[
4\operatorname{Cov}_{\pi_4}(U_{\rm src},U)
=4\operatorname{Var}_{\pi_4}(U_{\rm src})
 +4\operatorname{Cov}_{\pi_4}(U_{\rm src},U_{\rm old}).
\]
The last term has no fixed sign.  The fixed-cut bound
$U_{\rm old}\in L^1((a,T);L^\infty)$ alone does not control its
terminal time-weighted contribution without an estimate for
$U_{\rm src}$.  Any bound on the cumulative source response must retain
these inherited, transport, and diffusion terms, with only one physical
quartic viscous debit.

\section{Circulation barriers and geometric criteria}\label{sec:geometric}
Define
\[
 \mathrm{Pe}=-\frac{ru^r}{\nu}.
\]
\begin{proposition}\label{prop:peclet}
For $0<\alpha<2$ and $w(r)=r^\alpha$,
\begin{equation}\label{eq:peid}
 (\partial_t+b\cdot\nabla-\nu\Delta_*)w
 =\nu\alpha(2-\alpha-\mathrm{Pe})r^{\alpha-2},
 \qquad
 \Delta_*=\partial_r^2-r^{-1}\partial_r+\partial_z^2.
\end{equation}
Hence, on a cylinder $0<r<R$,
\[
 \mathrm{Pe}\le2-\alpha,
 \qquad
 |\Gamma|\le Cr^\alpha\ \text{on the parabolic boundary}
 \quad\Longrightarrow\quad
 |\Gamma|\le Cr^\alpha.
\]
\end{proposition}
\begin{proof}
Since
\[
 w_r=\alpha r^{\alpha-1},
 \qquad
 \Delta_*w=\alpha(\alpha-2)r^{\alpha-2},
\]
and $u^r=-\nu\mathrm{Pe}/r$,
\[
 b\cdot\nabla w=u^rw_r
 =-\nu\alpha\mathrm{Pe}\,r^{\alpha-2},
\]
which gives \eqref{eq:peid}.  On $\varepsilon<r<R$, the hypothesis
$\mathrm{Pe}\le2-\alpha$ makes $Cr^\alpha$ a supersolution of the
circulation equation.  Apply the parabolic maximum principle to
$\Gamma-Cr^\alpha$ and $-\Gamma-Cr^\alpha$.  Smooth axis compatibility gives
$\Gamma=r^2F=O(r^2)$; since $\alpha<2$, the inner boundary $r=\varepsilon$
is dominated for all sufficiently small $\varepsilon$.  Let
$\varepsilon\downarrow0$.
\end{proof}

The implication is one-way:
\[
 \mathrm{Pe}\le2-\alpha\Longrightarrow |\Gamma|\lesssim r^\alpha.
\]
Its failure is not a singularity criterion.  Z.~Zhang's weighted radial
criterion \cite{ZujinZhang2018} is used independently.

\subsection*{Weighted circulation carriers}
Fix $p>3/2$, $r_0>0$, and
\[
 L(r)=\log\frac{r_0}{r},
 \qquad 0<r<r_0/e.
\]
For a nonzero sign branch of $\Gamma$, put
\[
 Q=|\Gamma|L^p,
 \qquad
 M_p(t)=\sup_{0<r<r_0/e,\,z}Q(r,z,t),
 \qquad
 \mathrm{Pe}=-\frac{ru^r}{\nu}.
\]
At a nonzero maximizer the sign of $\Gamma$ is fixed locally, so the equation
for $|\Gamma|$ is the equation for $\pm\Gamma$.

\begin{proposition}\label{prop:carrier-equation}
Let $\Delta=\partial_r^2+\partial_z^2$ and
\[
 B_p=\left(u^r+\frac\nu r\left(1-\frac{2p}{L}\right),u^z\right),
 \qquad
 V_p=\frac{\nu p}{r^2L}
 \left(2-\mathrm{Pe}-\frac{p+1}{L}\right).
\]
Then
\begin{equation}\label{eq:carrier-Q}
 \partial_tQ+B_p\cdot\nabla Q-\nu\Delta Q+V_pQ=0.
\end{equation}
If $x_c(t)=(r_c(t),z_c(t))$ is an interior derivative-realizing maximizer of
$Q$, define
\[
 L_c=L(r_c),
 \qquad
 \kappa_c=-\frac{r_c^2\Delta Q(x_c,t)}{Q(x_c,t)}\ge0.
\]
Then, for almost every such time,
\begin{equation}\label{eq:carrier-ledger}
 D_t^+\log M_p
 =\frac\nu{r_c^2}
 \left[
 \frac p{L_c}\left(\mathrm{Pe}_c-2+\frac{p+1}{L_c}\right)-\kappa_c
 \right].
\end{equation}
If $\kappa_c>0$ and
\[
 \Sigma_c=
 \frac{p(\mathrm{Pe}_c-2+(p+1)/L_c)}{L_c\kappa_c},
\]
then
\begin{equation}\label{eq:carrier-sigma}
 D_t^+\log M_p=\frac{\nu\kappa_c}{r_c^2}(\Sigma_c-1).
\end{equation}
\end{proposition}
\begin{proof}
Write $\Gamma=QL^{-p}$ on the chosen sign branch.  Since
\[
 \partial_rL=-\frac1r,
 \qquad
 \partial_rL^{-p}=\frac{p}{rL}L^{-p},
\]
direct differentiation gives
\[
 \left(\partial_{rr}-\frac1r\partial_r+\partial_{zz}\right)(QL^{-p})
 =L^{-p}\left[
 \Delta Q+\frac1r\left(\frac{2p}{L}-1\right)Q_r
 +\frac{p(p+1-2L)}{r^2L^2}Q
 \right].
\]
Substitute this identity into
\[
 \Gamma_t+b\cdot\nabla\Gamma
 =\nu\left(\Gamma_{rr}-r^{-1}\Gamma_r+\Gamma_{zz}\right)
\]
and divide by $L^{-p}$.  The $Q_r$ terms combine into $B_p^rQ_r$, while
$u^r=-\nu\mathrm{Pe}/r$ gives \eqref{eq:carrier-Q}.

At an interior derivative-realizing maximum,
\[
 \nabla Q(x_c,t)=0,
 \qquad
 \Delta Q(x_c,t)\le0.
\]
Danskin's theorem therefore yields
\[
 D_t^+M_p=Q_t(x_c,t)
 =\nu\Delta Q(x_c,t)-V_p(x_c,t)Q(x_c,t).
\]
Divide by $Q(x_c,t)=M_p(t)$ and use the definition of $\kappa_c$ to obtain
\eqref{eq:carrier-ledger}.  Equation \eqref{eq:carrier-sigma} is algebraic.
\end{proof}

\begin{corollary}\label{cor:carrier-dichotomy}
Assume all sufficiently late maximizers relevant to unbounded growth of $M_p$
are interior.  If for some $\kappa_0,\varepsilon>0$,
\[
 \kappa_c\ge\kappa_0,
 \qquad
 \frac{\mathrm{Pe}_c}{L_c}\le\frac{\kappa_0}{p}-\varepsilon
\]
on every sufficiently late growth branch, then $M_p$ is bounded and the
solution is regular by Wei's logarithmic circulation criterion
\cite{Wei2016}.  Consequently, if regularity fails, then for every $p>3/2$ and every
$\kappa_0>0$, either growth carriers with $\kappa_c<\kappa_0$ occur
arbitrarily late, or
\begin{equation}\label{eq:carrier-dichotomy}
 \boxed{
 \limsup_{\substack{t\uparrow T:\ D_t^+\log M_p\ge0\\ \kappa_c\ge\kappa_0}}
 \frac{\mathrm{Pe}_c}{L_c}\ge\frac{\kappa_0}{p}.}
\end{equation}
\end{corollary}
\begin{proof}
Under the two displayed assumptions,
\[
 \frac p{L_c}\left(\mathrm{Pe}_c-2+\frac{p+1}{L_c}\right)-\kappa_c
 \le -p\varepsilon+o(1)<0.
\]
At every sufficiently late time with $D_t^+\log M_p>0$, the same estimate
would give $D_t^+\log M_p<0$, a contradiction.  Hence
$D_t^+M_p\le0$ almost everywhere terminally.  Since $M_p$ is locally
Lipschitz on every preterminal compact interval, it is terminally
nonincreasing.  Therefore $M_p$ is bounded, and
\[
 |\Gamma(r,z,t)|\le C L(r)^{-p}.
\]
Since $p>3/2$, after decreasing the fixed axis radius if necessary,
$CL^{-p}\le L^{-3/2}$, and Wei's criterion applies.  The alternative
\eqref{eq:carrier-dichotomy} is the contrapositive.  Branch reselection does
not change the envelope argument: $M_p$ is continuous, and Danskin's formula
applies almost everywhere to an active maximizing branch.
\end{proof}

\begin{proposition}\label{prop:carrier-budget}
Suppose $x_c(t)$ is a $C^1$ interior maximizing branch.  Put
\[
 \mathrm{slip}_r=\dot r_c-B_p^r(x_c,t).
\]
Then
\begin{equation}\label{eq:carrier-budget}
 \boxed{
 \frac d{dt}\log|\Gamma_c|
 =-\frac\nu{r_c^2}
 \left[
 \kappa_c+\frac p{L_c}\left(1+\frac{p-1}{L_c}\right)
 \right]
 +\frac{p\,\mathrm{slip}_r}{r_cL_c}.}
\end{equation}
\end{proposition}
\begin{proof}
Along the branch,
\[
 |\Gamma_c|=M_pL_c^{-p},
 \qquad
 \dot L_c=-\frac{\dot r_c}{r_c}.
\]
Hence
\[
 \frac d{dt}\log|\Gamma_c|
 =\frac d{dt}\log M_p+\frac{p\dot r_c}{r_cL_c}.
\]
Insert \eqref{eq:carrier-ledger}, use
$\mathrm{Pe}_c=-r_cu_c^r/\nu$, and write
\[
 \dot r_c-u_c^r
 =\mathrm{slip}_r+\frac\nu{r_c}\left(1-\frac{2p}{L_c}\right).
\]
The $\mathrm{Pe}_c$ term cancels, leaving \eqref{eq:carrier-budget}.
\end{proof}

\subsection*{A peak identity}
\begin{proposition}\label{prop:peak-slip}
Let $x_p(t)=(r_p(t),z_p(t))$ be a $C^1$ path of strict positive maxima of
$F(\cdot,t)$, with invertible Hessian.  Then
\begin{equation}\label{eq:peak-slip}
 \nabla^2F\,(\dot x_p-b)=2F\nabla U-\nu\nabla\Delta_5F.
\end{equation}
If
\[
 K_F=-\frac{\nabla^2F}{F_p}>0,
\]
then
\begin{equation}\label{eq:peak-slip-solved}
 \dot x_p-b
 =-2K_F^{-1}\nabla U
 +\frac\nu{F_p}K_F^{-1}\nabla\Delta_5F.
\end{equation}
\end{proposition}
\begin{proof}
Differentiate $\nabla F(x_p(t),t)=0$:
\[
 0=\nabla F_t+\nabla^2F\,\dot x_p.
\]
Using \eqref{eq:F},
\[
 F_t=-b\cdot\nabla F-2UF+\nu\Delta_5F.
\]
At the maximum,
\[
 \nabla(b\cdot\nabla F)=\nabla^2F\,b,
 \qquad
 \nabla(UF)=F\nabla U.
\]
Substitution gives \eqref{eq:peak-slip}; inversion gives
\eqref{eq:peak-slip-solved}.
\end{proof}

\begin{corollary}\label{cor:locked}
If
\[
 \dot x_p-b=o(1),
 \qquad
 \frac\nu{F_p}K_F^{-1}\nabla\Delta_5F=o(1),
 \qquad
 \|K_F\|+\|K_F^{-1}\|\le C,
\]
then
\[
 \nabla U(x_p(t),t)=o(1),
 \qquad
 \partial_zu^r(x_p(t),t)=o(r_p(t)).
\]
\end{corollary}
\begin{proof}
Rearrange \eqref{eq:peak-slip-solved} and use the three hypotheses.  Since
$u^r=rU$, $\partial_zu^r=rU_z$.
\end{proof}

\subsection*{Localized signed circulation response}
Consider
\begin{equation}\label{eq:normalized-circ}
 \gamma_s+\widetilde b\cdot\nabla\gamma
 =\bar\nu\left(\gamma_{\rho\rho}-\rho^{-1}\gamma_\rho+\gamma_{\zeta\zeta}\right)
\end{equation}
on $\delta<\rho<K$, with
\[
 |\gamma|\le1,
 \qquad
 |\gamma|\le2\rho^2,
 \qquad
 |\widetilde b|\le\bar b.
\]

\begin{proposition}\label{prop:signed-boundary-response}
Assume
\[
 \eta\gamma(\rho_*,0,0)\ge q,
 \qquad
 0<q\le1,
 \quad
 \eta\in\{-1,1\},
 \quad
 \rho_*\le\sqrt2.
\]
Choose
\[
 \delta^2=\frac q{16},
 \qquad
 \lambda=\max\left(\frac2\delta,\frac{4\bar b}{\bar\nu}\right),
\]
\[
 \kappa=\frac{\bar\nu\lambda^2}{8}e^{-\lambda(K-\delta)},
 \qquad
 S=\kappa^{-1}\log\frac{16e^{\lambda K}}q,
\]
\[
 L=\bar bS+\sqrt{4\bar\nu S\log(32/q)}.
\]
Let $\omega$ be the outer-boundary hitting measure of the backward diffusion
associated with \eqref{eq:normalized-circ}, stopped at $\rho=\delta$,
$\rho=K$, or time $-S$.  Then
\begin{equation}\label{eq:signed-response}
 \int_{\rho=K,\ |\zeta|\le L}\eta\gamma\,d\omega\ge\frac{5q}{8},
\end{equation}
and
\begin{equation}\label{eq:signed-response-weight}
 \omega\{\rho=K,\ |\zeta|\le L,\ \eta\gamma\ge q/2\}
 \ge\frac{q}{8-4q}\ge\frac q8.
\end{equation}
\end{proposition}
\begin{proof}
Let $a\ge0$ be backward elapsed time.  Before radial exit,
\[
 dR_a=\left[-\widetilde b^r(R_a,Z_a,-a)-\frac{\bar\nu}{R_a}\right]da
      +\sqrt{2\bar\nu}\,dW_a^r,
\]
\[
 dZ_a=-\widetilde b^z(R_a,Z_a,-a)\,da
      +\sqrt{2\bar\nu}\,dW_a^z.
\]
Let $\tau$ be the first exit from $\delta<\rho<K$ and
$\sigma=\min(\tau,S)$.  It\^o's formula and optional stopping give
\begin{equation}\label{eq:stop-representation}
 \eta\gamma(\rho_*,0,0)
 =\mathbb E\bigl[\eta\gamma(R_\sigma,Z_\sigma,-\sigma)\bigr].
\end{equation}

Set
\[
 \Psi(\rho)=2e^{\lambda K}-e^{\lambda\rho}.
\]
For the backward generator
\[
 \mathcal A_a=\bar\nu(\partial_{\rho\rho}-\rho^{-1}\partial_\rho+
 \partial_{\zeta\zeta})-\widetilde b(\cdot,-a)\cdot\nabla,
\]
the choices of $\lambda$ and $\kappa$ give
\[
 \mathcal A_a\Psi\le-\kappa\Psi.
\]
Therefore $e^{\kappa(a\wedge\tau)}\Psi(R_{a\wedge\tau})$ is a
supermartingale.  Since $\Psi\ge1$ and $\Psi(\rho_*)\le2e^{\lambda K}$,
\[
 \mathbb P(\tau>S)
 \le2e^{\lambda K}e^{-\kappa S}=\frac q8.
\]
On $\rho=\delta$,
\[
 |\gamma|\le2\delta^2=\frac q8.
\]
Moreover
\[
 |Z_a|\le\bar bS+\sqrt{2\bar\nu}\sup_{0\le s\le S}|W_s^z|.
\]
The reflection principle and the Gaussian tail estimate give
\[
 \mathbb P\left(\sup_{a\le\sigma}|Z_a|>L\right)
 \le4\exp\left[-\frac{(L-\bar bS)^2}{4\bar\nu S}\right]
 =\frac q8.
\]
Thus the old-time, inner-boundary, and far-axial contributions in
\eqref{eq:stop-representation} are each at most $q/8$.  The remaining signed
outer-boundary contribution is at least $5q/8$, proving
\eqref{eq:signed-response}.

Let $E=\{\eta\gamma\ge q/2\}$ and $p_E=\omega(E)$.  Since $|\gamma|\le1$,
\[
 \int\eta\gamma\,d\omega
 \le p_E+(1-p_E)\frac q2.
\]
Together with \eqref{eq:signed-response}, this gives
$p_E\ge q/(8-4q)$.
\end{proof}

\subsection*{Euler-length exit}
\begin{proposition}\label{prop:civ-exit}
Let $u$ be an axisymmetric suitable weak solution, smooth before $T_*$.  Assume
there is an Euler length $\ell(t)$ satisfying the hypotheses of
Constantin--Ignatova--Vicol \cite{CIVEulerLength2026} and, on the associated
shrinking neighborhood,
\[
 |u(x,t)|\le C\frac{\ell(t)}{T_*-t},
 \qquad
 |\nabla^2u(x,t)|\le\frac{C}{(T_*-t)\ell(t)}.
\]
Then the candidate point is regular.
\end{proposition}
\begin{proof}
This is the local axisymmetric theorem of
Constantin--Ignatova--Vicol \cite{CIVEulerLength2026}; all of its Euler-length
hypotheses are retained.
\end{proof}

\begin{proposition}\label{prop:hodge-civ}
Let
\[
 Q_H=|\partial_ru^r|+|u^r/r|+|\partial_zu^z|+|\partial_zu^r|,
 \qquad
 J_H=|\psi_z|+|r\psi_{rz}|+|r\psi_{zz}|.
\]
Then
\[
 \boxed{\frac12J_H\le Q_H\le4J_H.}
\]
\end{proposition}
\begin{proof}
Set $a=\psi_z$, $b=r\psi_{rz}$ and $c=r\psi_{zz}$.  Then
\[
 \frac{u^r}{r}=-a,
 \qquad
 \partial_ru^r=-(a+b),
 \qquad
 \partial_zu^z=2a+b,
 \qquad
 \partial_zu^r=-c.
\]
Hence
\[
 Q_H\le(|a|+|b|)+|a|+(2|a|+|b|)+|c|\le4J_H.
\]
Also
\[
 |a|+|c|\le Q_H,
 \qquad
 |b|\le|a+b|+|a|\le Q_H,
\]
so $J_H\le2Q_H$.
\end{proof}

\begin{corollary}\label{cor:hodge-prefix}
Let $u_N$ be exact smooth axisymmetric unforced solutions on $[0,t_N]$,
$t_N\uparrow T$, with $u_N(0)\to u_0$ in $H^s$, $s>5/2$.  If for some fixed
ball $B_\rho$ and $\delta>0$,
\[
 (T-t_N)\|J_H[u_N](t_N)\|_{L^\infty(B_\rho)}\ge\delta,
\]
then the exact solution from $u_0$ cannot remain smooth through $T$.
\end{corollary}
\begin{proof}
If the solution $u$ from $u_0$ were smooth on $[0,T]$, strong local
well-posedness and continuous dependence in $H^s$ would give
\[
 \sup_{0\le t\le T}\|u_N(t)-u(t)\|_{H^s}\to0.
\]
Since $s>5/2$, Sobolev embedding gives $H^s\hookrightarrow W^{1,\infty}$.
Therefore
\[
 \|J_H[u_N](t_N)-J_H[u](t_N)\|_{L^\infty(B_\rho)}\to0.
\]
Smoothness of $u$ makes $J_H[u]$ bounded near $T$, so the displayed
normalized quantity tends to zero, a contradiction.
\end{proof}

\subsection*{Circulation ledger}
\begin{proposition}\label{prop:absolute-circ-ledger}
For every nonnegative smooth cutoff $\phi$,
\begin{equation}\label{eq:absolute-circ-ledger}
 \frac{d}{dt}\int |\Gamma|\phi\,\frac{dr\,dz}{r}
 \le
 \int |\Gamma|\bigl[b\cdot\nabla\phi-2U\phi+\nu L_0\phi\bigr]
 \frac{dr\,dz}{r}
 -2\nu\int |F(0,z,t)|\phi(0,z)\,dz.
\end{equation}
For $\phi=\phi(t,r,z)$ add
\[
 \int|\Gamma|\phi_t\,\frac{dr\,dz}{r}
\]
to the right side.
\end{proposition}
\begin{proof}
Kato's inequality, applied to
$g_\delta=(\Gamma^2+\delta^2)^{1/2}-\delta$ and then
$\delta\downarrow0$, gives on $r>0$
\[
 \partial_t|\Gamma|+b\cdot\nabla|\Gamma|\le\nu L_0|\Gamma|.
\]
Integrate first over $r>\varepsilon$ with measure $dr\,dz/r$.
Axisymmetric incompressibility gives
\[
 \partial_ru^r+\partial_zu^z=-\frac{u^r}{r}=-U,
 \qquad
 \operatorname{div}_{r,z}(b/r)=-2U/r.
\]
Hence
\[
 -\int b\cdot\nabla|\Gamma|\,\phi\frac{dr\,dz}{r}
 =\int|\Gamma|(b\cdot\nabla\phi-2U\phi)\frac{dr\,dz}{r}.
\]
For diffusion, Green's identity for
$L_0=\partial_{rr}-r^{-1}\partial_r+\partial_{zz}$ yields
\[
 \int (L_0|\Gamma|)\phi\frac{dr\,dz}{r}
 =\int|\Gamma|L_0\phi\frac{dr\,dz}{r}
 -\int\frac{\partial_r|\Gamma|\,\phi-|\Gamma|\phi_r}{r}
 \Big|_{r=\varepsilon}dz.
\]
Since $\Gamma=r^2F$,
\[
 \frac{\partial_r|\Gamma|}{r}\to2|F(0,z,t)|,
 \qquad
 \frac{|\Gamma|\phi_r}{r}\to0.
\]
Let $\varepsilon\downarrow0$.  The Kato defect on interior nodal sets is
nonnegative and may be discarded.  If $\phi$ depends on time, differentiate
$|\Gamma|\phi$ and retain the $\phi_t$ term.
\end{proof}

Replacing $|\Gamma|$ by $\Gamma$ gives the signed identity; its axis term is
$-2\nu\int F(0,z,t)\phi(0,z)dz$ and has no fixed sign for arbitrary swirl.

\section{Source-memory estimates}\label{sec:memory}
\begin{proposition}\label{prop:causal-phase-aging-d18}
Fix a regular source cut $a$ and let $G_a$ solve
\[
 (\partial_t+b\cdot\nabla-\nu\Delta_5)G_a=j,
 \quad G_a(a)=0,\qquad j=\partial_z(F^2).
\]
Let $S(t,s)$ be its genuine transport--diffusion propagator,
$P=(-\Delta_5)^{-1}$, $q=(u^\theta)^2$, and
$V_*(t)=2\pi^2V_4(t)$. Set
\[
 \mathcal D j=(\partial_t+b\cdot\nabla-\nu\Delta_5)j,
 \qquad
 \mathfrak M_a(t)=\int_a^t(s-a)S(t,s)\mathcal D j(s)\,ds.
\]
On every compact regular time interval where the pairings are defined,
\begin{align}
 G_a(t)&=(t-a)j(t)-\mathfrak M_a(t),\label{eq:causal-phase-aging-d18}\\
 \mathcal W_a(t)&:=\frac{\langle G_a(t),[P,q(t)]j(t)\rangle_5}{V_*(t)}
 =-\frac{\langle\mathfrak M_a(t),[P,q(t)]j(t)\rangle_5}{V_*(t)}.
 \label{eq:phase-commutator-d18}
\end{align}
For $a<b<t$, with both source tags evolved in the \emph{same} actual
meridional drift,
\begin{equation}\label{eq:phase-recut-d18}
 \mathcal W_a(t)=\mathcal W_b(t)+
 \frac{\langle S(t,b)G_a(b),[P,q(t)]j(t)\rangle_5}{V_*(t)}.
\end{equation}
Moreover, with $U=u^r/r$ and $\Gamma=r^2F$, the exact source-production
identity is
\begin{align}
 \mathcal D j
 &=-4\partial_z(UF^2)-(\partial_zb)\cdot\nabla(F^2)
      -2\nu\partial_z|\nabla F|^2\notag\\
 &=\frac{2\Gamma}{r^4}
   \big[(rU_r-2U)\Gamma_z-rU_z\Gamma_r\big]
      -2\nu\partial_z|\nabla F|^2.
 \label{eq:phase-circ-d18}
\end{align}
For $r>0$, setting $\rho=r^2/2$, $A=U/r^2$ and $B=\Gamma^2$,
the inviscid part of \eqref{eq:phase-circ-d18} is precisely
$\{A,B\}_{\rho,z}=A_\rho B_z-A_zB_\rho$.
\end{proposition}
\begin{proof}
Since $\partial_s S(t,s)=-S(t,s)L_s$ for
$L_s=-b(s)\cdot\nabla+\nu\Delta_5$, differentiation yields
$\partial_s[S(t,s)j(s)]=S(t,s)\mathcal D j(s)$.
Integrating Duhamel's formula for $G_a$ and then integrating by parts
in source age gives \eqref{eq:causal-phase-aging-d18}. The commutator
$[P,q]$ is skew-adjoint, so $\langle j,[P,q]j\rangle_5=0$,
which proves \eqref{eq:phase-commutator-d18}.
The re-cut formula follows from
$G_a(t)=S(t,b)G_a(b)+G_b(t)$.
Finally $\mathcal D(F^2)=-4UF^2-2\nu|\nabla F|^2$,
while $[\mathcal D,\partial_z]f=-(\partial_zb)\cdot\nabla f$.
These give the first line of \eqref{eq:phase-circ-d18}; substituting
$F=\Gamma/r^2$ and
$\partial_z u^z=-2U-rU_r$ gives its second line.
The change of variables $\partial_\rho=r^{-1}\partial_r$
gives the displayed Jacobian.
\end{proof}

The phase term $\mathcal W_a$ is a component of the complete
quartic-exchange residual. Its backward receiver depends on the reporting
terminal time, and the re-cut identity \eqref{eq:phase-recut-d18}
retains the propagated earlier current.

Fix a core scale $R$, a source chamber with $r\asymp R$, and a source scale
$d\le R/4$.  Let
\[
 (\partial_t+B\cdot\nabla-\nu\Delta_5)N=\partial_z(\chi F^2),
 \qquad N(t_0)=0,
\]
with source support of diameter $O(d)$.  Assume
\begin{equation}\label{eq:kernel-gradient}
 \int|\partial_{y_z}K(t,x;s,y)|\,dy
 \le\frac{Ce^{C\mathcal A_{\rm def}}}{\sqrt{\nu(t-s)}},
 \qquad
 \mathcal A_{\rm def}=\int_{t_0}^t\|\nabla B\|_\infty,
\end{equation}
and
\begin{equation}\label{eq:chamber-hodge}
 R|b[g](x)|\le C_HdR^2\|g\|_\infty.
\end{equation}

\begin{proposition}\label{prop:fresh-filter}
If $t-t_0\le C_0d^2/\nu$, then
\[
 \boxed{
 R|b_N(x)|\le
 Ce^{C\mathcal A_{\rm def}}
 \frac{\Gamma_*^2}{\nu}\left(\frac dR\right)^2.}
\]
\end{proposition}
\begin{proof}
Duhamel's formula and integration by parts in $y_z$ give
\[
 N(t,x)
 =-\int_{t_0}^t\int\partial_{y_z}K(t,x;s,y)\,\chi F^2(y,s)\,dy\,ds.
\]
Using \eqref{eq:kernel-gradient} and $|F|\le C\Gamma_*/R^2$ on $r\asymp R$,
\[
 \|N(t)\|_\infty
 \le Ce^{C\mathcal A_{\rm def}}\frac{\Gamma_*^2}{R^4}
 \int_{t_0}^t\frac{ds}{\sqrt{\nu(t-s)}}.
\]
The last integral equals $2\sqrt{(t-t_0)/\nu}\le C d/\nu$.  Thus
\[
 \|N(t)\|_\infty
 \le Ce^{C\mathcal A_{\rm def}}
 \frac{\Gamma_*^2d}{\nu R^4}.
\]
Apply \eqref{eq:chamber-hodge}.
\end{proof}

\begin{proposition}\label{prop:memory}
Let cohorts $N_j$ be born at scales $R_j$ on $O(R_j^2/\nu)$ clocks.  Assume
uniform deformation, spatial tightness on $O(R_j)$, and no unresolved finer
source scale.  Then, for $R_k\le R_j$,
\[
 R_k|b_j|\le C\frac{\Gamma_*^2}{\nu}\frac{R_k}{R_j},
 \qquad
 R_k^2|U_j|\le C\frac{\Gamma_*^2}{\nu}
 \left(\frac{R_k}{R_j}\right)^2.
\]
If $R_{j+1}\le qR_j$, $0<q<1$, then
\[
 \sum_{j\le k}\frac{R_k}{R_j}\le\frac1{1-q}.
\]
\end{proposition}
\begin{proof}
The birth estimate gives
\[
 \|N_j\|_\infty\le C\frac{\Gamma_*^2}{\nu R_j^3}.
\]
Spatial tightness and the five-dimensional Newtonian kernel give
\[
 |b_j|\le CR_j^2\|N_j\|_\infty,
 \qquad
 |U_j|\le CR_j\|N_j\|_\infty.
\]
Therefore
\[
 R_k|b_j|\le C\frac{\Gamma_*^2}{\nu}\frac{R_k}{R_j},
 \qquad
 R_k^2|U_j|\le C\frac{\Gamma_*^2}{\nu}
 \left(\frac{R_k}{R_j}\right)^2.
\]
Finally $R_k/R_j\le q^{k-j}$, and the geometric series gives the claim.
\end{proof}

Pan's theorem \cite{Pan2016} imposes a meridional-drift condition; bounded
circulation alone does not activate it.  Accordingly, the compact-memory step
below remains an explicit hypothesis.

\begin{remark}\label{rem:age-tail}
If
\[
 |S_{nj}(E)|\le C\tau_n\kappa_j2^{j-n},
\]
then
\[
 \sum_{j\le n-A}|S_{nj}(E)|
 \le C\tau_n\left(\sup_{j\le n-A}\kappa_j\right)2^{1-A}.
\]
No fixed-age conclusion follows without control of the prefactor.  The scalar
choice
\[
 \tau_n=n,
 \qquad
 \kappa_j=1,
 \qquad
 j_n=n-\lfloor\log_2n\rfloor
\]
satisfies
\[
 \tau_n2^{j_n-n}\in[1,2),
 \qquad
 n-j_n\to\infty.
\]

For a variable cutoff $A_n$, a sufficient condition for the discarded tail to be negligible at the terminal logarithmic scale is
\[
 \sum_{n<N}\tau_n\sum_{j\le n-A_n}\kappa_j2^{j-n}
 =o\!\left(\log\frac{T-a}{T-t_N}\right),
\]
together with control of the unfinished terminal record. Thus $A_n$ may diverge, or retain the entire post-$a$ history.
\end{remark}

\section{Radial dissipation and causal source bounds}\label{sec:radial-lorentz}
Let $v=u^\theta$ and use the reduced physical measure $d\mu=r\,dr\,dz$.
The results of this section hold on regular intervals of the exact
unforced axisymmetric flow. They do not assert the signed continuation
condition in \eqref{eq:signed-source-intro}.

\begin{proposition}\label{prop:radial-floor}
Suppose $v\not\equiv0$ and $M_0=\|rv(a)\|_\infty$ is finite and nonzero.
Put $E_\theta(t)=\int |v(t)|^2\,d\mu$ and
\[
 D_{4,\perp}(t)=\frac34\int|\partial_r(v(t)^2)|^2\,d\mu
      +\int\frac{|v(t)|^4}{r^2}\,d\mu .
\]
Then $D_4\ge D_{4,\perp}$, and at every later regular time,
\begin{equation}\label{eq:radial-floor}
 D_{4,\perp}(t)\ge C_{\rm rad}
 \frac{V_4(t)^2}{M_0^2E_\theta(t)},
 \qquad C_{\rm rad}=\left(1+\frac{\sqrt3}{2}\right)^2
 =\frac74+\sqrt3.
\end{equation}
Moreover, for every $R>0$,
\begin{equation}\label{eq:quartic-radial-tail}
 \int_{r\ge R}|v(t)|^4\,d\mu
 \le \frac{M_0^2}{R^2}E_\theta(t).
\end{equation}
\end{proposition}
\begin{proof}
Set $g=v^2$, $A=\sqrt3/2$, and
$m_2=\int r^2g^2\,d\mu$. For $\lambda>0$, expanding
$\int|Ag_r-g/r+\lambda rg|^2\,d\mu\ge0$ gives
\[
 D_{4,\perp}+\lambda^2m_2
 \ge2\lambda(A+1)V_4.
\]
Indeed $\int g_rg/r\,d\mu=0$ and
$\int rgg_r\,d\mu=-V_4$ by radial integration by parts.
Optimizing $\lambda$ yields
$D_{4,\perp}m_2\ge(A+1)^2V_4^2$.
The circulation maximum principle from \eqref{eq:Gamma} gives
$|rv(r,z,t)|\le M_0$, and therefore
\[
 m_2=\int(rv)^2v^2\,d\mu\le M_0^2E_\theta.
\]
This proves \eqref{eq:radial-floor}; the tail estimate follows directly
from $|v|^4=(rv)^2|v|^2/r^2$. Smooth axial compatibility
and decay justify the integrations, and approximation gives the stated
form inequality. Since
$3|v|^2|\partial_rv|^2=\tfrac34|\partial_r(v^2)|^2$,
the definition in \eqref{eq:quartic-dissipation-def} gives
$D_4\ge D_{4,\perp}$.
\end{proof}

\begin{proposition}\label{prop:axial-uncertainty-d18}
In addition to the notation of \cref{prop:radial-floor}, suppose
$V_4(t)>0$ and the axial second moment of $|v(t)|^4$ is finite. Put
\[
 h_4(t)^2=\inf_{Z\in\mathbb R}\frac{1}{V_4(t)}
       \int (z-Z)^2|v(t)|^4\,d\mu,
 \qquad
 D_{4,z}(t)=\frac34\int|\partial_z(v(t)^2)|^2\,d\mu.
\]
Then
\begin{equation}\label{eq:axial-uncertainty-d18}
 \frac{D_{4,z}}{V_4}\ge\frac{3}{16h_4^2},
 \qquad
 \frac{D_4}{V_4}\ge
 C_{\rm rad}\frac{V_4}{M_0^2E_\theta}+\frac{3}{16h_4^2}.
\end{equation}
In particular, whenever $V_4'(t)\ge0$,
\[
 \left\|\left(\frac{u^r}{r}\right)_-\right\|_\infty
 \ge\nu\left(C_{\rm rad}\frac{V_4}{M_0^2E_\theta}
                        +\frac{3}{16h_4^2}\right).
\]
\end{proposition}
\begin{proof}
For $g=v^2$, integration in $z$ gives
$\int g^2\,d\mu=-2\int(z-Z)gg_z\,d\mu$.
Cauchy--Schwarz, followed by the infimum over $Z$, yields
$\int|g_z|^2\,d\mu\ge V_4/(4h_4^2)$.
The first estimate follows by multiplication by $3/4$.
The identity $D_4=D_{4,\perp}+D_{4,z}$ and
\eqref{eq:radial-floor} give the second; the quartic balance
$V_4'=-4\int Uv^4\,d\mu-4\nu D_4$ gives the last assertion.
The two sectors partition the \emph{existing} viscous debit.
\end{proof}

The axial variance is a \emph{global} quartic width; it need not
measure a localized near-circulation-cap ring in the presence of
distant quartic mass. No estimate for the signed source--receiver
interaction follows from \eqref{eq:axial-uncertainty-d18} alone.

\begin{corollary}\label{cor:radial-compression-tax}
Under the hypotheses of \cref{prop:radial-floor}, write
$U_- =(-u^r/r)_+$. Then
\begin{equation}\label{eq:radial-tax}
 V_4'(t)\le4\|U_-(t)\|_\infty V_4(t)
 -\frac{4\nu C_{\rm rad}V_4(t)^2}{M_0^2E_\theta(t)}.
\end{equation}
Consequently at each regular time at which $V_4'(t)\ge0$,
\begin{equation}\label{eq:radial-compressor-threshold}
 \|U_-(t)\|_\infty\ge
 \frac{\nu C_{\rm rad}V_4(t)}{M_0^2E_\theta(t)}.
\end{equation}
\end{corollary}
\begin{proof}
The exact quartic identity is
\[
 V_4'=-4\int U|v|^4\,d\mu-4\nu D_4.
\]
Estimate the first term by $4\|U_-\|_\infty V_4$ and
apply \cref{prop:radial-floor}. The second assertion follows by
rearranging when $V_4'\ge0$.
\end{proof}

The lower bound \eqref{eq:radial-floor} belongs to the \emph{same}
quartic debit already present in the signed source action, not an
independent dissipative budget. No optimality is claimed for the combined circulation bound
\eqref{eq:radial-floor}. Neither
\eqref{eq:radial-compressor-threshold} nor
\eqref{eq:quartic-radial-tail} gives an upper bound for $U_-$ or
excludes a migrating, increasingly thin swirl concentration.

For the next proposition, use \emph{full} Lebesgue measures on
$\mathbb R^3$ and $\mathbb R^5$. The lifted variable $F=v/r$ is
viewed as a radial function of $y\in\mathbb R^4$, with $r=|y|$.
\begin{proposition}\label{prop:lorentz-port}
Let $u$ be a smooth finite-energy unforced solution on $[a,b]$ and put
$E(a)=\frac12\|u(a)\|_{L^2(\mathbb R^3)}^2$. Then
\begin{equation}\label{eq:five-dimensional-fisher}
 \|\nabla_5F(t)\|_{L^2(\mathbb R^5)}^2
 =\pi\|\nabla(v(t)e_\theta)\|_{L^2(\mathbb R^3)}^2,
\end{equation}
and, for a dimensional constant $C_{5,L}$,
\begin{equation}\label{eq:lorentz-port}
 \nu\int_a^b\|F(t)^2\|_{L^{5/3,1}(\mathbb R^5)}\,dt
 \le\pi C_{5,L}E(a).
\end{equation}
No global circulation bound is needed for this estimate.
\end{proposition}
\begin{proof}
With $v=rF$, the physical angular-gradient identity is
\[
\int_{\mathbb R^3}|\nabla(ve_\theta)|^2dx
 =2\pi\int_0^\infty\!\int_{\mathbb R}
  \bigl(v_r^2+v_z^2+v^2/r^2\bigr)r\,dz\,dr.
\]
Expanding $v_r=F+rF_r$ and integrating
$2r^2FF_r= r^2\partial_r(F^2)$ by parts cancels the
remaining $2rF^2$ term. The last display becomes
$2\pi\int r^3|\nabla_{r,z}F|^2\,dr\,dz$.
Since the area of $S^3$ is $2\pi^2$, this is precisely
\eqref{eq:five-dimensional-fisher}.
The physical energy identity now yields
$\nu\int_a^b\|\nabla_5F\|_2^2dt\le\pi E(a)$.
By the standard five-dimensional Sobolev--Lorentz embedding
$\dot H^1(\mathbb R^5)\hookrightarrow L^{10/3,2}(\mathbb R^5)$,
\[
 \|F^2\|_{L^{5/3,1}}\le C_{5,L}\|\nabla_5F\|_2^2.
\]
Integrating proves \eqref{eq:lorentz-port}.
\end{proof}

The source norm in \eqref{eq:lorentz-port} is paid by the
\emph{existing} kinetic-energy dissipation. It cannot be counted as
additional dissipation in \eqref{eq:signed-source-intro}.
The next identity records the exact causal placement that is still
uncontrolled by that source norm.

\begin{proposition}\label{prop:source-ordered-dual}
Fix one regular source cut $a$. Let $\mathcal S_b(t,s)$ be the true
propagator for $\partial_t+b\cdot\nabla-\nu\Delta_5$, with the
meridional drift of the solution, and put
\[
 f(s)=F(s)^2,\quad H=-\partial_z(-\Delta_5)^{-1},\quad
 W(t)=r^2F(t)^4,\quad \mathcal V(t)=\int_{\mathbb R^5}W(t)\,dX.
\]
For $a\le s\le\tau<T$ define
\[
 p_\tau(s)=\mathcal S_b(\tau,s)^*H^*
            \frac{W(\tau)}{\mathcal V(\tau)},\qquad
 Q_t(s)=\int_s^t\partial_zp_\tau(s)\,d\tau.
\]
Then the actual source-generated cumulative compression obeys
\begin{equation}\label{eq:causal-lorentz-identity}
 C_a(t):=-\int_a^t
 \frac{\int U_{\rm src}(\tau)|v(\tau)|^4\,d\mu}{V_4(\tau)}\,d\tau
 =\int_a^t\langle F(s)^2,Q_t(s)\rangle_{\mathbb R^5}\,ds.
\end{equation}
In particular the signed action of \cref{thm:signed-source-intro}
remains
\begin{equation}\label{eq:causal-action-once}
 N_a(t)=C_a(t)-\nu\int_a^tD_4(\tau)/V_4(\tau)\,d\tau.
\end{equation}
\end{proposition}
\begin{proof}
By Duhamel's formula and \eqref{eq:G},
$G_{\rm src}(\tau)=\int_a^\tau\mathcal S_b(\tau,s)
\partial_zf(s)\,ds$ and $U_{\rm src}=HG_{\rm src}$.
The reduced $r\,dr\,dz$ pairing for $U_{\rm src}v^4$ is the same
as the reduced $r^3dr\,dz$ pairing for $U_{\rm src}W$.
The ratio removes the common angular constants. Thus
\[
-\frac{\int U_{\rm src}(\tau)|v(\tau)|^4\,d\mu}{V_4(\tau)}
=-\int_a^\tau\langle \partial_zf(s),p_\tau(s)\rangle_{\mathbb R^5}\,ds
=\int_a^\tau\langle f(s),\partial_zp_\tau(s)\rangle\,ds.
\]
Integrate in $\tau$ and reverse time integration on the finite
regular triangle $a\le s\le\tau\le t$.
\end{proof}

The estimate \eqref{eq:lorentz-port} bounds the time integral of the
source norm, whereas \eqref{eq:causal-lorentz-identity} requires a
signed overlap with the actual backward terminal receiver. If
$Q_t(s)\in L^{5/2,\infty}(\mathbb R^5)$, Lorentz duality gives the
\emph{conditional} bound
\[
 C_a(t)\le C\int_a^t\|F(s)^2\|_{L^{5/3,1}}
                  \|[Q_t(s)]_+\|_{L^{5/2,\infty}}\,ds.
\]
No corresponding all-history bound for this true, report-dependent
dual follows from \eqref{eq:lorentz-port}. In particular, the
source-age integrand $\langle F(s)^2,Q_t(s)\rangle$ is not
$\frac{d}{dt}N_a(t)$: changing $t$ also changes the terminal receiver.
Neither a static frozen resolvent nor positive-part optimization
supplies that missing estimate.

\begin{proposition}\label{prop:weighted-axial-variation}
Let $u$ be a smooth finite-energy unforced axisymmetric solution on
$[a,b]$, and set $v=u^\theta$, $F=v/r$ and
$E(t)=\tfrac12\|u(t)\|_{L^2(\R^3)}^2$. Then, with physical measure
$dx=2\pi r\,dr\,dz$,
\begin{equation}\label{eq:weighted-axial-variation}
 \nu\int_a^b\!\int_{\R^3}r|\partial_z(F^2)|\,dx\,ds
 \le E(a)-E(b)\le E(a).
\end{equation}
The power one of $r$ cannot be replaced by any smaller nonnegative
power in a uniform estimate funded only by $E(a)$.
\end{proposition}
\begin{proof}
The exact pointwise identity and Young's inequality give
\[
 r|\partial_z(F^2)|
 =2\left|\frac vr\,\partial_zv\right|
 \le \frac{v^2}{r^2}+|\partial_zv|^2.
\]
The right-hand side is a part of
$|\nabla(ve_\theta)|^2$; integrate and apply the ordinary kinetic
energy identity. For the sharpness assertion, choose a nontrivial smooth
whole-space pure-swirl datum for which $\partial_z(F^2)$ is not
identically zero, and restrict to a short interval of smooth evolution.
Under $u_\lambda(x,t)=\lambda u(\lambda x,\lambda^2t)$, its energy scales
as $\lambda^{-1}$, whereas the spacetime integral with weight
$r^\alpha$ scales as $\lambda^{-\alpha}$. Their ratio diverges for
$\alpha<1$. This compares different data, not one singular history.
\end{proof}

\begin{proposition}\label{prop:causal-square-radial}
Use the fixed cut and genuine terminal adjoints of
\cref{prop:source-ordered-dual}. For $a<t<T$ put
\[
 \Pi_t(s)=2\pi^2\int_s^t p_\tau(s)\,d\tau,\quad
 w(s)=v(s)\partial_zv(s),\quad k(s)=\frac{3\nu}{V_4(s)}.
\]
Assume $V_4(s)>0$ and
$\int_a^t V_4(s)\|\Pi_t(s)\|_{L^2(d\mu)}^2ds<\infty$,
and the associated finite source-adjoint pairings are justified. Define
\begin{align*}
 \mathcal H_a(t)&=\int_a^t\frac{V_4(s)}{3\nu}
             \int|\Pi_t(s)|^2\,d\mu\,ds,\\
 \mathcal E_{{\rm sq},a}(t)&=\int_a^t\frac{3\nu}{V_4(s)}
       \int\left|w(s)+\frac{V_4(s)}{3\nu}\Pi_t(s)\right|^2
             d\mu\,ds,\\
 \mathcal E_{\perp,a}(t)&=\nu\int_a^t
                 \frac{D_{4,\perp}(s)}{V_4(s)}\,ds.
\end{align*}
Then the original source-only signed action has the exact decomposition
\begin{equation}\label{eq:causal-square-radial}
 \boxed{N_a(t)=\mathcal H_a(t)-\mathcal E_{{\rm sq},a}(t)
                         -\mathcal E_{\perp,a}(t).}
\end{equation}
All three terms on the right are nonnegative separately. If also
$0<M_0=\|rv(a)\|_\infty<\infty$, then
\begin{equation}\label{eq:causal-square-radial-floor}
 N_a(t)\le\mathcal H_a(t)-\mathcal E_{{\rm sq},a}(t)
 -\frac{\nu C_{\rm rad}}{M_0^2}
       \int_a^t\frac{V_4(s)}{E_\theta(s)}\,ds.
\end{equation}
\end{proposition}
\begin{proof}
In \cref{prop:source-ordered-dual}, reverse the axial derivative:
$C_a(t)=-\int_a^t\!\int \partial_z(F^2)
[\int_s^tp_\tau(s)d\tau]dX\,ds$.
Since $\partial_z(F^2)=2v\partial_zv/r^2$ and
$dX=2\pi^2r^3dr\,dz$, the normalization $d\mu=rdr\,dz$ gives
\[
 C_a(t)=-2\int_a^t\!\int w(s)\Pi_t(s)\,d\mu\,ds.
\]
The axial part of the original, single quartic debit is
$D_{4,z}=3\int w^2\,d\mu$; its remaining part is exactly
$D_{4,\perp}$ from \cref{prop:radial-floor}. Applying
$-2wp-kw^2=p^2/k-k|w+p/k|^2$ pointwise proves
\eqref{eq:causal-square-radial}. The lower bound for
$\mathcal E_{\perp,a}$ follows from
\eqref{eq:radial-floor}. No second viscosity or energy payment occurs.
\end{proof}

The quantity $\mathcal H_a$ is an unrestricted squared-adjoint
amplitude, not a signed production estimate; $\mathcal E_{{\rm sq},a}$
retains the actual source--receiver phase mismatch. The receiver
$\Pi_t$ depends on the \emph{reporting time} $t$, so this identity must
not be read as a differential inequality in terminal time. In particular,
\eqref{eq:causal-square-radial-floor} is not known to satisfy the
exponential--Ces\`aro condition of \cref{thm:signed-source-intro}.
The additional square-integrability assumption in
\cref{prop:causal-square-radial} is not imposed in the leading theorem.

For comparison with velocity-space regularity criteria, let $b^0$ be
the globally regular no-swirl Navier--Stokes solution starting from
$b(a)$, with $b=u^re_r+u^ze_z$. The exact difference $w_b=b-b^0$ obeys
\[
 (\partial_t-\nu\Delta)w_b
 +\mathbb P\operatorname{div}
 (b^0\otimes w_b+w_b\otimes b^0+w_b\otimes w_b)
 =-\mathbb P\operatorname{div}(v^2e_\theta\otimes e_\theta),
 \quad w_b(a)=0.
\]
When $M_0<\infty$, Cauchy--Schwarz, circulation control and the kinetic
energy identity also give
\begin{equation}\label{eq:angular-L6L3}
 \int_a^b\|v(s)e_\theta\|_{L^3(\R^3)}^6\,ds
 \le \frac{2M_0^2 E(a)^2}{\nu}.
\end{equation}
This estimate concerns the angular component; the meridional response
$w_b$ enters separately in the full-velocity $L^3$ criterion
\cite{Seregin2012}.

\section{Complete histories and finite prefixes}\label{sec:history}
Let
\[
 \Pi_0f=\frac1{2\pi}\int_0^{2\pi}\mathcal T_\alpha f\,d\alpha,
 \qquad
 \Qop=I-\Pi_0,
\]
and
\[
 \Rres[V]=\partial_tV-\nu\Delta V+\Pp((V\cdot\nabla)V).
\]
For a smooth divergence-free history $V$ on $[0,t]$, set
\[
 \cE[V;u_0]^2
 =\|V(0)-u_0\|_2^2
 +\nu^{-1}\|\Rres[V]\|_{L^2(0,t;\dot H^{-1})}^2,
 \qquad
 \cA[V;u_0]=\frac{\|\Qop V(t)\|_2}{\cE[V;u_0]}.
\]

\begin{proposition}\label{prop:horizontal}
For a smooth finite-energy solution, with
$Y=\|\nabla u\|_2^2$ and $H=\|\nabla_hu\|_2$,
\[
 Y'+\nu\|\Delta u\|_2^2\le C\nu^{-3}H^4Y.
\]
Hence $H\in L^4(0,T)$ implies continuation.  If $u$ is axisymmetric and
smooth on every $t<T$, then
\[
 u\text{ extends through }T
 \iff
 \sup_{t<T}\|\partial_1u(t)\|_2<\infty.
\]
\end{proposition}
\begin{proof}
Pair \eqref{eq:NS} with $-\Delta u$.  One integration by parts gives
\[
 \frac12Y'+\nu\|\Delta u\|_2^2
 =-\sum_{i,j,k}\int \partial_ku_i\,\partial_i u_j\,\partial_ku_j\,dx.
\]
For $k=1,2$ a horizontal derivative is explicit.  For $k=3$, either $i=1,2$
or, when $i=3$, incompressibility gives
\[
 \partial_3u_3=-\partial_1u_1-\partial_2u_2.
\]
Thus H\"older's inequality yields
\[
 \left|\int(u\cdot\nabla)u\cdot\Delta u\right|
 \le C\|\nabla_hu\|_2\|\nabla u\|_4^2.
\]
Gagliardo--Nirenberg gives
\[
 \|\nabla u\|_4^2
 \le C\|\nabla u\|_2^{1/2}\|\Delta u\|_2^{3/2}
 =CY^{1/4}\|\Delta u\|_2^{3/2}.
\]
Young's inequality with exponents $4/3$ and $4$ gives
\[
 C H Y^{1/4}\|\Delta u\|_2^{3/2}
 \le\frac\nu2\|\Delta u\|_2^2+C\nu^{-3}H^4Y.
\]
Absorb and apply Gronwall.  For an axisymmetric field,
\[
 \|\partial_1u\|_2=\|\partial_2u\|_2,
 \qquad
 H^2=2\|\partial_1u\|_2^2.
\]
\end{proof}

\begin{lemma}\label{lem:relative}
If the exact solution $u$ from $u_0$ extends through $T$, then every complete
smooth divergence-free history $V$ on $[0,t]$, $t<T$, satisfies
\[
 \|V(t)-u(t)\|_2\le e^{K_u(T)}\cE[V;u_0],
 \qquad
 K_u(T)=\int_0^T\kappa_u(s)\,ds,
\]
where
\[
 \kappa_u=\operatorname*{ess\,sup}
 [-\lambda_{\min}(\operatorname{Sym}\nabla u)]_+.
\]
Hence $\cA[V;u_0]\le e^{K_u(T)}$.
\end{lemma}
\begin{proof}
Let $e=V-u$.  Then
\[
 \partial_te-\nu\Delta e+(V\cdot\nabla)e+(e\cdot\nabla)u+\nabla q
 =\Rres[V].
\]
Taking the $L^2$ pairing with $e$ removes the pressure and transport terms.
The symmetric strain satisfies
\[
 \int e^T(\operatorname{Sym}\nabla u)e\,dx
 \ge-\kappa_u\|e\|_2^2.
\]
By $\dot H^{-1}$--$\dot H^1$ duality and Young's inequality,
\[
 |\langle\Rres[V],e\rangle|
 \le\frac\nu2\|\nabla e\|_2^2
 +\frac1{2\nu}\|\Rres[V]\|_{\dot H^{-1}}^2.
\]
Therefore
\[
 \frac d{dt}\|e\|_2^2
 \le2\kappa_u\|e\|_2^2
 +\nu^{-1}\|\Rres[V]\|_{\dot H^{-1}}^2.
\]
Gronwall gives
\[
 \|e(t)\|_2^2
 \le e^{2K_u(T)}\left(
 \|e(0)\|_2^2
 +\nu^{-1}\|\Rres[V]\|_{L^2_t\dot H^{-1}}^2\right).
\]
Since $\Qop u(t)=0$, $\|\Qop V(t)\|_2\le\|e(t)\|_2$.
\end{proof}

\begin{proof}[Proof of \cref{thm:mixed-prefix-intro}]
Assume instead that the exact strong solution $u$ from $u_0$ extends
smoothly through $T$.  For a complete history $V_n$, put $e=V_n-u$.
With $\kappa_u$ defined in \cref{lem:relative}, the relative-energy
calculation with the split residual gives
\[
 \frac12\frac d{dt}\|e\|_2^2+\frac\nu2\|\nabla e\|_2^2
 \le\kappa_u(t)\|e\|_2^2
  +\|R_{a,n}\|_2\|e\|_2
  +\frac1{2\nu}\|R_{b,n}\|_{\dot H^{-1}}^2.
\]
Since $\int_0^T\kappa_u<\infty$, an integrating factor and a
quadratic inequality for $\sup_{t\le t_n}\|e(t)\|_2$ give
\begin{equation}\label{eq:mixed-stability}
 \sup_{0\le t\le t_n}\|V_n(t)-u(t)\|_2
 \le C_{u,T}\varepsilon_n.
\end{equation}
The fixed compact detector rescales as a bounded linear functional on
$L^2$ with norm $\|\Phi\|_2/(\nu\sqrt{\ell_n})$:
\[
 \left|\frac{\ell_n}{\nu}\int e(t_n,x_n+\ell_n y)
                   \cdot\Phi(y)\,dy\right|
 \le\frac{\|e(t_n)\|_2\|\Phi\|_2}{\nu\sqrt{\ell_n}}=o(1).
\]
On the other hand, the smooth reference field satisfies
\[
 \left|\frac{\ell_n}{\nu}\int u(t_n,x_n+\ell_n y)
                   \cdot\Phi(y)\,dy\right|
 \le\frac{\ell_n}{\nu}\|\Phi\|_1
             \sup_{[0,T]}\|u(t)\|_\infty=o(1).
\]
This contradicts the fixed lower bound in
\eqref{eq:mixed-prefix-intro}.  The statement therefore follows
without requiring an exact geometric return or a source-born
successor.  The energy-class requirements in the theorem justify all
integration by parts and cutoff limits; no uniform higher-Sobolev bound on
$V_n$ is needed.
\end{proof}

\begin{proposition}\label{prop:finite-event-cost}
Fix $\nu>0$.  Suppose one smooth compactly supported axisymmetric
pure-swirl datum $w_0$ has a smooth, unforced, whole-space solution $w$
on $[0,s_2]$ and there are $0<s_1<s_2$ and $\gamma>0$ such that
\begin{equation}\label{eq:seed-positive-net}
 N_0^w(s)=\frac14\log\frac{V_4[w(s)]}{V_4[w_0]}
 \ge2\gamma \qquad (s_1\le s\le s_2).
\end{equation}
Let $E_w=\tfrac12\|w_0\|_2^2>0$ and
$\Gamma_w=\|r w_0^\theta\|_\infty>0$.  For each prescribed
$E_*>0$ there is a smooth compact pure-swirl datum $\bar u_0$ with
kinetic energy $E_*$ and a family of smooth compact pure-swirl data
$u_{0,\lambda}$ with exactly the same energy $E_*$, circulation maximum
$\Gamma_w$, and viscosity $\nu$, such that their \emph{exact}
solutions satisfy, for all sufficiently large $\lambda$,
\begin{equation}\label{eq:finite-event-cost}
\begin{gathered}
 N_0^{u_\lambda}(s/\lambda^2)\ge\gamma
 \quad (s_1\le s\le s_2),\qquad
 \nu\int_0^{s_2/\lambda^2}\|\nabla u_\lambda\|_2^2dt
 \le C\lambda^{-1},\\
 \|u_{0,\lambda}-\bar u_0\|_2
 =\sqrt{2E_w}\lambda^{-1/2}(1+o(1)).
\end{gathered}
\end{equation}
Thus a single positive source-funded, viscosity-compensated event need not
consume a fixed positive \emph{unnormalized} energy or dissipation amount
at fixed initial energy and circulation.  This conclusion is conditional
on \eqref{eq:seed-positive-net}; it gives neither recurrent events on
one trajectory nor a singular C2 family.
\end{proposition}
\begin{proof}
Choose a smooth compact pure-swirl spectator $\bar u_0$ supported in a
distant annulus, with kinetic energy $E_*$ and circulation maximum at most
$\Gamma_w/4$.  It can carry the prescribed energy while keeping that
maximum small by extending its smooth axial support.  For large $\lambda$
its support is disjoint from that of $\lambda w_0(\lambda x)$.  Set
\[
 c_\lambda=\left(1-\frac{E_w}{\lambda E_*}\right)^{1/2},
 \qquad
 u_{0,\lambda}(x)=\lambda w_0(\lambda x)+c_\lambda\bar u_0(x).
\]
Physical disjointness and the exact Navier--Stokes scaling give
$E[u_{0,\lambda}]=E_*$ and the common circulation maximum
$\Gamma_w$.  They also give
$\|u_{0,\lambda}-\bar u_0\|_2^2
 =2E_w/\lambda+2E_*(c_\lambda-1)^2$.

Rescale the \emph{actual} solution by
$\widetilde u_\lambda(s,y)=\lambda^{-1}
 u_\lambda(s/\lambda^2,y/\lambda)$.
Its initial value is $w_0+r_\lambda$, where
$r_\lambda(y)=c_\lambda\lambda^{-1}\bar u_0(y/\lambda)$ and
\[
 \|r_\lambda\|_\infty=O(\lambda^{-1}),
 \qquad \|r_\lambda\|_4=O(\lambda^{-1/4}).
\]
The full Leray heat kernel satisfies
$\|e^{\nu t\Delta}\Pp\nabla\cdot A\|_p
 \le C(\nu t)^{-1/2}\|A\|_p$
for $p=4,\infty$: its kernel is a first Gaussian derivative plus the
third derivative of the integrated Gaussian heat kernel, and both have
$L^1$ norm bounded by $C(\nu t)^{-1/2}$.  The mild difference equation
about the fixed smooth $w$ therefore admits a finite-interval Volterra
bootstrap, giving
\[
 \sup_{0\le s\le s_2}
 (\|\widetilde u_\lambda(s)-w(s)\|_\infty
 +\|\widetilde u_\lambda(s)-w(s)\|_4)
 \longrightarrow0.
\]
In particular these are exact smooth solutions through $s_2$ for large
$\lambda$; no independently evolved velocity fields are superposed.
Since all initial data are pure swirl, their homogeneous old-current tags
vanish and the exact quartic balance gives
$N_0^{u_\lambda}=\tfrac14\log(V_4(t)/V_4(0))$.
Scaling yields
$V_4[u_\lambda(s/\lambda^2)]=\lambda
 V_4[\widetilde u_\lambda(s)]$.
Continuity in $L^4$ transfers the strict margin in
\eqref{eq:seed-positive-net} to the first inequality in
\eqref{eq:finite-event-cost}.

Finally let $\bar u(t)$ be the smooth solution from $\bar u_0$ and
$e=u_\lambda-\bar u$ on $[0,s_2/\lambda^2]$.  Relative energy gives
\[
 \frac12\frac d{dt}\|e\|_2^2
 +\nu\|\nabla e\|_2^2
 \le\|\nabla\bar u\|_\infty\|e\|_2^2.
\]
The initial squared error is $O(\lambda^{-1})$, so Gronwall bounds its
energy and integrated dissipation by $C\lambda^{-1}$.  The fixed
spectator's dissipation over this shrinking interval is
$O(\lambda^{-2})$; the second inequality follows by the triangle
inequality for gradients.  The limit datum $\bar u_0$ is smooth but
$\|u_{0,\lambda}-\bar u_0\|_2$ is of order
$\sqrt{\ell_\lambda}$ for $\ell_\lambda=\lambda^{-1}$,
not $o(\nu\sqrt{\ell_\lambda})$, and the observation times tend to
zero.  The construction therefore does not satisfy
\cref{thm:mixed-prefix-intro}.
\end{proof}

\begin{theorem}\label{thm:complete-history}
Fix nonzero divergence-free axisymmetric $u_0\in H^3(\mathbb R^3)$ and let
$u$ be its unique strong solution, smooth on compact subintervals of $[0,T)$.
Then
\[
 \boxed{
 u\text{ fails to continue through }T
 \iff
 \exists\ t_N\uparrow T,\ V_N:
 \ \cE[V_N;u_0]\to0,
 \ \cA[V_N;u_0]\to\infty.}
\]
\end{theorem}
\begin{proof}
The forward implication of regularity is \cref{lem:relative}.  Conversely,
if continuation fails, \cref{prop:horizontal} gives $t_N\uparrow T$ with
\[
 \|\partial_1u(t_N)\|_2\to\infty.
\]
For $h\ne0$ set $V_{N,h}(s,x)=u(s,x+he_1)$.  Translation invariance gives
$\Rres[V_{N,h}]=0$.  Moreover,
\[
 \frac{u_0(\cdot+he_1)-u_0}{h}\to\partial_1u_0
 \quad\text{in }L^2,
\]
and, because $u(t_N)$ is axisymmetric,
\[
 \Qop\frac{u(t_N,\cdot+he_1)-u(t_N)}h
 \to\partial_1u(t_N)
 \quad\text{in }L^2.
\]
Choose $h_N\downarrow0$ so both difference quotients are within $1$ of their
limits.  Then
\[
 \cE[V_{N,h_N};u_0]\le h_N(\|\partial_1u_0\|_2+1)\to0,
\]
while
\[
 \cA[V_{N,h_N};u_0]
 \ge\frac{\|\partial_1u(t_N)\|_2-1}{\|\partial_1u_0\|_2+1}\to\infty.
\]
\end{proof}

\begin{proposition}\label{prop:translation-test}
Let $v_N$ be complete smooth axisymmetric divergence-free histories on
$[0,t_N]$ and set
\[
 d_N^2=\|v_N(0)-u_0\|_2^2
 +\nu^{-1}\|\Rres[v_N]\|_{L^2(0,t_N;\dot H^{-1})}^2,
\]
\[
 W_N=\|\partial_1v_N(t_N)\|_2,
 \qquad
 Z_N=\|\partial_{11}v_N(t_N)\|_2,
 \qquad
 A_0=\|\partial_1u_0\|_2.
\]
If $W_N,Z_N>0$ and $h_N=W_N/Z_N$, then
\[
 \cA[V_N;u_0]
 \ge\frac1{2(A_0/W_N+d_NZ_N/W_N^2)},
\]
where $V_N(t,x)=v_N(t,x+h_Ne_1)$.  Thus
\[
 W_N\to\infty,
 \qquad
 \frac{d_NZ_N}{W_N^2}\to0
\]
imply noncontinuation when $t_N\uparrow T$.
\end{proposition}
\begin{proof}
Since $v_N$ is axisymmetric, $\Qop v_N(t_N)=0$.  Taylor's formula in $L^2$
gives
\[
 \tau_{h_N}v_N-v_N
 =h_N\partial_1v_N
 +\int_0^{h_N}(h_N-s)\partial_{11}v_N(\cdot+se_1)\,ds.
\]
Therefore
\[
 \|\Qop V_N(t_N)\|_2
 \ge h_NW_N-\frac{h_N^2}{2}Z_N
 =\frac{W_N^2}{2Z_N}.
\]
Translation equivariance preserves the residual norm, and
\[
 \|\tau_{h_N}u_0-u_0\|_2\le h_NA_0.
\]
Minkowski's inequality in the direct sum defining $\cE$ gives
\[
 \cE[V_N;u_0]\le d_N+h_NA_0.
\]
Divide and use $h_N=W_N/Z_N$.
\end{proof}

\subsection*{Full-state renewal}
For a pure-swirl entrance, the equations give
\[
 G(0)=0,
 \quad
 \partial_z(F_0^2)\not\equiv0
 \quad\Longrightarrow\quad
 G_t(0)=\partial_z(F_0^2)\not\equiv0,
\]
\[
 U=-\partial_z(-\Delta_5)^{-1}G,
\]
\[
 \mathrm{Pe}\ge2\ \text{at isolated points}
 \not\Longrightarrow
 \text{persistent material or full-state return}.
\]
Thus a successor must inherit the generated meridional current and the
nonlocal exterior state; resetting to pure swirl changes the solution.

\subsection*{Finite-prefix criterion}
The complete-history criterion gives
\[
 W_N\to\infty,
 \qquad
 \frac{d_NZ_N}{W_N^2}\to0
 \stackrel{\cref{prop:translation-test}}{\Longrightarrow}
 \text{noncontinuation},
\]
and
\[
 (T-t_N)\|J_H[u_N](t_N)\|_{L^\infty(B_\rho)}\ge\delta
 \stackrel{\cref{cor:hodge-prefix}}{\Longrightarrow}
 \text{noncontinuation of the limit solution}.
\]
Thus exact self-copy is unnecessary.

Let $\ell_n=\ell_0\theta^n$, $0<\theta<1$, and assume on stage $n$
\[
 \|u(t)\|_\infty\le K\frac\nu{\ell_n},
 \qquad
 \Delta t_n\le C\frac{\ell_n^2}{\nu}.
\]
Then
\[
 T-t_n\le\sum_{m\ge n}\Delta t_m
 \le\frac{C\ell_n^2}{\nu(1-\theta^2)},
\]
so
\[
 \sqrt{\frac{T-t}{\nu}}\,\|u(t)\|_\infty
 \le K\sqrt{\frac{C}{1-\theta^2}}.
\]
Such a geometric return is Type I.  A singular construction must leave this
bounded-state, bounded-parabolic-clock regime.

\medskip\noindent\emph{Energy error and receiver fidelity.}
The $L^2$ error in \cref{thm:complete-history} tests the final
symmetry-breaking witness through the relative-energy argument; it does
not control the quartic receiver $V_4$ on intermediate states.  To see the distinction, choose nonzero
$\varphi\in C_c^\infty((1,2)\times(-1,1))$ and let
\[
 e_\ell(r,z)=\frac\nu\ell\varphi(r/\ell,z/\ell)e_\theta,
 \qquad 0<\ell<\tfrac12.
\]
These are smooth, compactly supported divergence-free fields with
$\|r e_\ell^\theta\|_\infty\le C\nu$, and scaling gives
$\|e_\ell\|_2^2=c_2\nu^2\ell$ and
$V_4(e_\ell)=c_4\nu^4/\ell$, where $c_2,c_4>0$ are fixed.
For $v_\ell=e_\ell+e_{\ell^2}$ the radial supports are disjoint, so
\[
 \frac{\|v_\ell-e_\ell\|_2}{\|e_\ell\|_2}
 =\sqrt\ell\longrightarrow0,
 \qquad
 \frac{V_4(v_\ell)}{V_4(e_\ell)}
 =1+\ell^{-1}\longrightarrow\infty.
\]
This is a comparison of admissible smooth states, not a Navier--Stokes
trajectory.  An iterative construction must separately control its
quartic receiver or prove downstream stability that prevents hidden
subscale packets; final $L^2$ accuracy alone is insufficient.

\section{Necessary conditions for a first singular time}\label{sec:necessary}
The following statements are contrapositives of the continuation
criteria. Those involving $\Gamma_*$ are asserted only when
$0<\Gamma_*<\infty$; the signed source condition retains its separate
strong-data hypotheses.

\begin{theorem}\label{thm:necessary-blowup}
Let $T<\infty$ be a first singular time and $a<T$ a fixed regular cut.
With the notation of \cref{thm:cumulative-compression-intro,thm:cumulative-M-intro,thm:quartic-dose-intro,thm:source-inflow-intro},
\begin{align}
 \limsup_{t\uparrow T}\frac{\int_a^t\sigma(q)\,dq}{s(t)}
 &\ge \frac12+\frac{\nu^2}{6\Gamma_*^2},\label{eq:nec-compression}\\
 \limsup_{t\uparrow T}\frac{\nu^{-1}\int_a^tM(q)^2\,dq}{s(t)}
 &\ge\frac16,\label{eq:nec-angular}\\
 \limsup_{t\uparrow T}\frac{\nu^{-7}\int_a^tV_4(q)^2\,dq}{s(t)}
 &\ge\frac{3}{2C_4}.\label{eq:nec-quartic}
\end{align}
For every fixed $R>0$,
\begin{equation}\label{eq:nec-source}
 \sup_{I\subset[a,T)}
 \frac{\int_I g_R^{\rm src}(t)\,dt}{\sqrt{\nu|I|}}=\infty.
\end{equation}
For the signed quartic receiver defined in \eqref{eq:F4},
\begin{equation}\label{eq:nec-F4}
 \limsup_{S\to\infty}\frac1S\int_0^S[\mathscr F_4(\xi)]_+\,d\xi
 \ge\frac18.
\end{equation}
The free first current obeys \eqref{eq:first-current-bound}; it cannot by
itself supply a super-elapsed-Type-I contribution.
\end{theorem}
\begin{proof}
If any of the three strict reverse inequalities in
\eqref{eq:nec-compression}--\eqref{eq:nec-quartic} held, the corresponding
continuation theorem would contradict the choice of $T$.
Finiteness in \eqref{eq:nec-source} for one fixed $R$ would give the
hypothesis of \cref{thm:source-inflow-intro}.  A strict reverse
inequality in \eqref{eq:nec-F4} would give continuation by
\cref{thm:q4receiver}.  The last assertion is
\cref{thm:first-current-intro}.
\end{proof}

\begin{corollary}\label{cor:signed-source-necessary}
Under the strong Sobolev hypotheses of
\cref{thm:signed-source-intro}, if $T$ is a first singular time then,
for every fixed regular cut $a<T$,
\[
 \limsup_{t\uparrow T}
 \frac{\displaystyle\int_a^t e^{8N_a(q)}\,dq}
      {\displaystyle\log\frac{T-a}{T-t}}>0,
 \qquad
 \int_a^T Z_{\rm src}(q)^2\,dq=\infty.
\]
No global circulation bound is required for this strong-class implication.
\end{corollary}
\begin{proof}
Otherwise one of the sufficient conclusions of
\cref{thm:signed-source-intro,cor:sourced-enstrophy-intro} would give
smooth continuation through $T$.
\end{proof}

\begin{corollary}\label{cor:radial-record-necessary}
Assume the hypotheses of \cref{prop:radial-floor} at a regular cut
$a<T$, and suppose that $T$ is the first singular time. There are
regular times $t_n\uparrow T$ with $V_4(t_n)\to\infty$ and
$V_4'(t_n)>0$. Along such a sequence,
\[
 \left\|\left(\frac{u^r(t_n)}r\right)_-\right\|_\infty
 \ge\frac{\nu C_{\rm rad}}{M_0^2\|u(a)\|_{L^2(\mathbb R^3)}^2}
 V_4(t_n)\longrightarrow\infty.
\]
For each fixed $R>0$, the fraction of quartic swirl mass in
$\{r\ge R\}$ tends to zero along the sequence.
\end{corollary}
\begin{proof}
If $V_4$ were bounded on $[a,T)$, then
$\int_a^TV_4(t)^2\,dt<\infty$, which meets the sufficient condition
of \cref{thm:quartic-dose-intro}, a contradiction. Thus $V_4$ is
unbounded, and its continuity and differentiability on each regular
compact interval give arbitrarily high upcrossing times with
$V_4'>0$. Apply \eqref{eq:radial-compressor-threshold} and the
physical energy bound
$E_\theta(t)\le\|u(a)\|_{L^2(\mathbb R^3)}^2$.
The radial-mass assertion follows from
\eqref{eq:quartic-radial-tail} after division by $V_4(t_n)$.
\end{proof}

These conditions are necessary, not sufficient for singularity.  At equality in \eqref{eq:nec-compression} or \eqref{eq:nec-angular}, the endpoint logarithmic continuation allowances of \cref{thm:cumulative-compression-intro,thm:cumulative-M-intro} must also fail.

A singular solution must also violate the hypotheses of each applicable regularity theorem: the Type-I and Liouville criteria \cite{CSTY2008,KNSS2009,SereginSverak2009}, logarithmic circulation criteria \cite{LeiZhang2017,Wei2016}, slightly supercritical meridional criteria \cite{Pan2016}, angular Prodi--Serrin criteria \cite{ChenFangZhang2017}, the one-sided radial Type-I criterion \cite{Zhang2026}, and the Euler-length criteria \cite{CIVEulerLength2026}.  These exclusions do not, without further hypotheses, imply multiscale source ancestry.

None of our direct criteria prescribes a near-axis modulus for $\Gamma$.
They bound cumulative quantities rather than pointwise values.  In particular,
\cref{thm:source-inflow-intro} controls only the radial velocity generated after
$a$, not the total radial velocity in Q.~S.~Zhang's criterion.  No inclusion
between our criteria and the Euler-length criteria of
Constantin--Ignatova--Vicol is claimed.
With $\|\Gamma(a)\|_\infty<\infty$, a uniform bound on
$\|r(u^re_r+u^ze_z)\|_\infty$ implies $|u|\le C/r$.
Under the hypotheses of the global $C/r$ criteria
\cite{CSTY2009II,KNSS2009}, it therefore suffices for continuation.
The reverse implication is not asserted: unbounded $r|b|$ at spatial
infinity may occur for smooth, regular fields.

The continuation criteria do not classify every possible terminal
source history. In particular, no global finite-age, source-compactness,
or first-use packing hypothesis is assumed.

\section{Discussion}\label{sec:relation}
For a fixed regular cut $a<T$, use the signed source action $N_a$
defined before \cref{thm:signed-source-intro}, including the single
physical quartic viscous debit. By
\cref{thm:signed-source-intro}, the signed condition (A) below is sufficient
without a global circulation cap in the strong $H^m$ class.  The
source-enstrophy bound in \cref{cor:sourced-enstrophy-intro} yields an
equivalent continuation condition for actual strong histories, but its
unsigned \emph{a priori} estimate is stronger and generally less economical
than the signed criterion (A).  The remaining routes can be written as
\[
 \tag{A}
 \int_a^t e^{8N_a(q)}\,dq
 =o\!\left(\log\frac{T-a}{T-t}\right),
\]
\[
 \tag{B}
 \int_I\sup_{0<r<R,z}(-u^r_{\rm src})_+\,dt
 \lesssim\sqrt{\nu|I|}
 \quad\text{for every terminal }I,
\]
or, on the bounded-record branch, a same-history upper response contradicting
\eqref{eq:signed-response}; and
\[
 \tag{C}
 \text{the complete mixed-residual and detector hypotheses of}
 \ \cref{thm:mixed-prefix-intro}.
\]
(A) or (B) gives regularity.  Condition (C) proves noncontinuation.  A physical full-state Type-II
successor is one possible way to construct such histories, but exact
self-copy or a prescribed relay profile is not required.

The auxiliary estimates in \cref{sec:radial-lorentz} isolate the radial
and axial parts of the single quartic debit, energy-funded source norms,
and the exact source--receiver completed square. The causal phase-aging
identity \eqref{eq:phase-commutator-d18} further isolates the effect of
a changing source history. None supplies a uniform signed bound for
the true terminal adjoint or establishes (A). The newborn circulation-curvature identity
\eqref{eq:newborn-curvature-r96} and the example
\cref{prop:newborn-positive-example-r96} show why a universal
negative newborn-sign hypothesis cannot supply the missing estimate.

The conditional finite-event construction in
\cref{prop:finite-event-cost} also shows that charging each
productive source episode an absolute positive energy cost cannot, by
itself, rule out a singularity.  Its solutions have different initial
data and the source events occur at shrinking initial times; it proves
nothing about repeated events on one actual terminal history.  In particular
the common-datum, scale-relative error condition in (C) remains necessary
for this constructive test.

The weighted-circulation carrier estimates give additional conditional
geometric constraints under their stated hypotheses, but are not a
classification of singular profiles. A hypothetical singular solution
must violate every applicable continuation criterion, including the
classical Type-I, circulation, radial-inflow, and Euler-length criteria.

\par\smallskip\noindent

\appendix
\section{Positive newborn source production}\label{app:positive-newborn-r96}

\begin{proposition}\label{prop:newborn-positive-example-r96}
There exists smooth, finite-energy, divergence-free, axisymmetric
pure-swirl initial data on $\mathbb R^3$, with bounded circulation, for
which the actual source-generated quartic stretching rate satisfies
\[
 \sigma_{\rm src}(0)=0,\qquad \sigma_{\rm src}'(0)>0.
\]
Thus no universal nonpositive sign holds for the newborn source term.
\end{proposition}
\begin{proof}
Put $x=r^2$ and define
\begin{align*}
 w_1(x)&=(x-2)e^{-x},
 & f_1(x)&=(-4x^2+25x-26)e^{-x},\\
 f_0(x)&=(x^9+4x^2+25x+27)e^{-x}.
\end{align*}
These smooth radial functions obey
\[
 f_0+f_1=(x^9+50x+1)e^{-x}>0,\qquad
 f_0-f_1=(x^9+8x^2+53)e^{-x}>0.
\]
Hence $f_0(r^2)+f_1(r^2)\cos z>0$. Let
$B=-\partial_r^2-3r^{-1}\partial_r$ be the nonnegative radial
Laplacian on $H=L^2((0,\infty),r^3\,dr)$. Direct differentiation gives
\[
 (B+1)w_1(r^2)=f_1(r^2).
\]
For the formal $2\pi$-periodic profile $F^2=f_0+f_1\cos z$,
the multiplier $\mathcal K=-\partial_{zz}(-\Delta_5)^{-1}$ therefore
satisfies $\mathcal K(F^2)=w_1\cos z$.
Since $u^\theta=rF$, the source-born derivative is
$V_4\sigma_{\rm src}'(0)=-\int r^2F^4\mathcal K(F^2)\,dm_5$.
Integrating over one axial period gives
\[
 V_4\sigma_{\rm src}'(0)
  =-2\pi\int_0^\infty r^5 f_0(r^2)f_1(r^2)w_1(r^2)\,dr.
\]
All integrals are elementary after $x=r^2$. In particular,
\begin{align*}
 \int_0^\infty r^5f_0f_1w_1\,dr
 &=\frac{12908}{729}-\frac{7884800}{59049}
 =-\frac{6839252}{59049}<0,\\
 \int_0^\infty r^5f_1w_1\,dr&=\frac78.
\end{align*}
The periodic newborn numerator is consequently $2\pi\cdot 6839252/59049>0$ in the
reduced measure. Periodicity itself does not give finite energy in
$\mathbb R^3$, so we next localize the axial variable.

Choose $0\le\eta\in C_c^\infty(\mathbb R)$ with $\eta\not\equiv0$;
for $\varepsilon>0$ define
\[
 F_\varepsilon(r,z)=\eta(\varepsilon z)
  \sqrt{f_0(r^2)+f_1(r^2)\cos z},\qquad
 u_{0,\varepsilon}=rF_\varepsilon\,e_\theta.
\]
This datum is smooth in Cartesian coordinates, divergence free,
axis-compatible, and rapidly decaying, hence belongs to every
$H^m(\mathbb R^3)$. Its circulation $r^2F_\varepsilon$ is bounded.
Local strong well-posedness provides the corresponding genuine
Navier--Stokes evolution on a positive time interval, so the
source-generated derivative at $t=0$ is well defined.

It remains to check that localization preserves the sign. Write
$h=\eta^2$, $h_\varepsilon(z)=h(\varepsilon z)$ and
$K_\xi=\xi^2(B+\xi^2)^{-1}$ on $H$. For all real $\xi$,
$\|K_\xi\|_{H\to H}\le1$. The only radial solutions of $Bf=0$
are $f=A+Cr^{-2}$, so $\ker B=\{0\}$ in $H$. Consequently,
$K_\xi f_0\to0$ in $H$ as $\xi\to0$;
$K_\xi f_1\to K_1f_1=w_1$ as $\xi\to\pm1$.
Plancherel, with the substitutions $\xi=\varepsilon\zeta$
and $\xi=\pm1+\varepsilon\zeta$, and dominated convergence against
$|\widehat h(\zeta)|^2$, give
\begin{align*}
 \sqrt\varepsilon\,\|\mathcal K(h_\varepsilon f_0)\|_{L^2(dm_5)}&\longrightarrow0,\\
 \sqrt\varepsilon\,\|\mathcal K(h_\varepsilon f_1\cos z)
             -h_\varepsilon w_1\cos z\|_{L^2(dm_5)}&\longrightarrow0.
\end{align*}
Here the products use radial functions of $r^2$; the notation is
suppressed for readability. Meanwhile
$\|r^2 F_\varepsilon^4\|_{L^2(dm_5)}=O(\varepsilon^{-1/2})$.
By Cauchy--Schwarz, all multiplier errors vanish after multiplying
the newborn numerator by $\varepsilon$. The remaining periodic factor
averages in the slowly varying axial envelope, so
\begin{align*}
 \lim_{\varepsilon\downarrow0}
 \varepsilon V_{4,\varepsilon}\sigma_{\rm src}'(0)
 &=-\left(\int_{\mathbb R}\eta(Z)^6\,dZ\right)
    \int_0^\infty r^5 f_0f_1w_1\,dr\\
 &=\frac{6839252}{59049}\int_{\mathbb R}\eta(Z)^6\,dZ>0.
\end{align*}
Thus some sufficiently small $\varepsilon>0$ gives
$\sigma_{\rm src}'(0)>0$ for actual smooth finite-energy whole-space
data. This sign statement concerns the infinitesimal source-generated
quartic work only, not the full viscous debit or a singularity.
\end{proof}

\end{document}